\documentclass[english,10.5pt]{article}
\usepackage{amsmath}
\usepackage{amssymb}
\usepackage{verbatim}
\usepackage{amsthm}
\usepackage{mathrsfs}
\usepackage{graphicx}
\usepackage{mathtools}
\usepackage[colorlinks,pdfdisplaydoctitle,linkcolor=blue,citecolor=red]{hyperref}
\usepackage{caption}
\usepackage[top=3.cm, bottom=3cm, right=2.85cm, left=2.85cm]{geometry}
\usepackage{subcaption}
\usepackage{url}
\usepackage{epstopdf}
\usepackage{pgf,tikz}
\usepackage{todonotes}
\usepackage{bbm}
\usepackage{extarrows}
\usetikzlibrary{arrows}
\usepackage{algorithm}
 \usepackage{algorithmicx}
\usepackage{amssymb,amsthm}    
\usepackage{graphicx}          
\usepackage{amsmath}         
\usepackage{caption}
\usepackage{subcaption}
\usepackage{dsfont}
\usepackage{listings}
\usepackage{bbm}
\usepackage{color}
\usepackage{ dsfont }
\usepackage{enumerate}
 \usepackage{relsize}
\usepackage{xfrac}
\renewcommand{\Delta}{\triangle}

\definecolor{darkblue}{rgb}{0,0,0.7}

\definecolor{darkgreen}{rgb}{0.01,0.75,0.24}

\def \Ee[#1]{\mathcal{E}^{\text{{#1}}}}

\def\pa[#1,#2]{\frac{\partial {#1}}{\partial {#2}} }

\def\idom[#1,#2,#3]{\int_{#1}\hspace{1pt} {#2} \hspace{1pt} \text{d}{#3}}
\def\res[#1,#2]{\left.{#1}\right|_{#2}}

\def\var[#1,#2]{\langle \delta \mathcal{E}^{\text{{#1}}}({#2}),v\rangle}
\def\vars[#1,#2,#3]{\langle \delta^2\mathcal{E}^{\text{{#1}}}({#2})v,{#3}\rangle}
\def\vard[#1,#2,#3,#4]{\langle \delta\mathcal{E}^{\text{{#1}}}({#2})-\delta\mathcal{E}^{\text{{#3}}}({#4}),v\rangle}

\usepackage{hyperref}

\newcommand{\be}{\begin{equation}}
\newcommand{\en}{\end{equation}}
\newcommand{\ben}{\begin{equation*}}
\newcommand{\enn}{\end{equation*}}
\newcommand{\bea}{\begin{aligned}}
\newcommand{\ena}{\end{aligned}}

\def\ba#1\ena{\begin{align}#1\end{align}}
\def\ban#1\enan{\begin{align*}#1\end{align*}}

\theoremstyle{plain}
\newtheorem{thm}{Theorem}[section]
\newtheorem{theorem}{Theorem}[section]

\newtheorem{lem}[thm]{Lemma}

\newtheorem{proposition}[thm]{Proposition}

\newtheorem{remark}[thm]{Remark}

\numberwithin{equation}{section}

\begin{document}
\begin{center}
{\Large \textbf{Multilevel Ensemble Kalman-Bucy Filters}}

\vspace{0.5cm}

BY NEIL K. CHADA, AJAY JASRA \& FANGYUAN YU

{\footnotesize Computer, Electrical and Mathematical Sciences and Engineering Division, \\ King Abdullah University of Science and Technology, Thuwal, 23955, KSA.} \\
{\footnotesize E-Mail:\,} \texttt{\emph{\footnotesize neil.chada@kaust.edu.sa, ajay.jasra@kaust.edu.sa, fangyuan.yu@kaust.edu.sa}}
\end{center}

%
%

\begin{abstract}
In this article we consider the linear filtering problem in continuous-time. {We} develop and apply multilevel Monte Carlo (MLMC) strategies for ensemble Kalman--Bucy filters (EnKBFs). These filters can be viewed as approximations of  conditional McKean--Vlasov-type diffusion processes. They are also interpreted as the continuous-time analogue of the \textit{ensemble Kalman filter}, which has proven to be successful due to its applicability and computational cost. We prove that {an ideal version of} our multilevel EnKBF can achieve a mean square error (MSE) of $\mathcal{O}(\epsilon^2), \ \epsilon>0$ with a cost of order $\mathcal{O}(\epsilon^{-2}\log(\epsilon)^2)$. In order to prove this result we provide a Monte Carlo convergence and approximation bounds associated to time-discretized EnKBFs. This implies a reduction in cost compared to the (single level) EnKBF which requires a cost of $\mathcal{O}(\epsilon^{-3})$ to achieve an MSE of $\mathcal{O}(\epsilon^2)$. We test our theory on a linear problem, which we motivate through {high-dimensional examples of order $\sim \mathcal{O}(10^4)$ and $\mathcal{O}(10^5)$}. 
\end{abstract}


\section{Introduction}
\label{sec:intro}

The filtering problem refers to the recursive estimation of an unobserved Markov process associated to  sequentially observed data.
It is ubiquitous in many applications found in disciplines such as applied mathematics, statistics and engineering; see for instance \cite{BC09,CR11,PD04,LSZ15}.
Mathematically speaking, in filtering we have an unobserved $d_x$--dimensional Markov process of interest $\{X_t\}_{t \geq 0}$, and a $d_y$--dimensional observed process $\{Y_t\}_{t \geq 0}$, which are defined through the following stochastic differential equations
\begin{align}
\label{eq:dat}
dY_t &=  h(X_t) dt +  dV_t,   \\
\label{eq:sig}
dX_t &=   f(X_t) dt +  \sigma(X_t) dW_t, 
\end{align}
where $V_t$ and $W_t$ are independent $d_y-$ and $d_x-$dimensional Brownian motions respectively, with $h:\mathbb{R}^{d_x} \rightarrow \mathbb{R}^{d_y}$ and $f:\mathbb{R}^{d_x} \rightarrow \mathbb{R}^{d_x}$ denoting potentially nonlinear functions, and $\sigma: \mathbb{R}^{d_x} \rightarrow  \mathbb{R}^{d_x \times d_x}$ acting as a diffusion coefficient. The aim of the filtering problem is to compute a conditional expectation $\mathbb{E}[\varphi(X_t)|\mathscr{F}_t]$, where $\varphi:\mathbb{R}^{d_x} \rightarrow \mathbb{R}$ is an appropriately integrable function and $\{\mathscr{F}_t\}_{t \geq 0}$ is the filtration generated by the observed process \eqref{eq:dat}.  

In most problems of practical interest, the filtering distribution is intractable. A notable exception is in the linear case
which can be solved via the celebrated Kalman filter (KF) \cite{REK60}. It was first proposed by Rudolph Kalman and since then has been a highly-applicable algorithm. Popular extensions of the KF include the extended Kalman filter (ExKF) which motivates the KF in a nonlinear setting, and the ensemble Kalman filter (EnKF) \cite{GE09,GE94} which provides a Monte Carlo version of the KF, which is more computationally feasible in high-dimensional settings. For both the discrete and continuous-time filtering problem, the EnKF generates $N\in\mathbb{N}$ dependent samples (or particles) in parallel and recursively in time. The computational savings, relative to the KF, are due to the fact that the covariance matrix is not updated at each time step, {but instead an ensemble of particles are updated. This updated ensemble of particles is then used to estimate both the sample mean and covariances, which is repeated up to a finite iteration count}. If we assume that $h$ and $f$ are linear, then 
a variety of convergence results exist for the EnKF in both {discrete \cite{GMT11,MCB11,XTT18}} and continuous time \cite{DT18}. {Some of these works also apply to some nonlinear problems, however with the EnKF the main interest arises in linear filtering.}

One recent development in applied mathematics which can improve upon Monte Carlo methods is the multilevel Monte Carlo (MLMC) method. This was developed in \cite{MBG08,MBG15,hein} whose motivation was to ensure that the mean square error (MSE) was of order $\mathcal{O}(\epsilon^2)$ (for some $\epsilon>0$), but with a reduction in computational cost relative to using an ordinary Monte Carlo method. The approach is typically used for problems that are subject to discretization, such as the approximation of expectations w.r.t.~a law of a diffusion process, for instance by using the Euler method. The idea is to write the expectation w.r.t.~the law of a very precise discretization as a collapsing sum of expectations w.r.t.~laws of increasingly coarse discretizations. Then by appropriately sampling from couplings of laws of discretizations it is possible to reduce the cost to achieve an MSE of $\mathcal{O}(\epsilon^2)$ relative to an ordinary Monte Carlo method.
Since its original application for diffusion processes in finance, it has been applied to filtering methods for uncertainty quantification. Most notably this includes multilevel particle filters \cite{JKL17} and the multilevel ensemble Kalman filter (MLEnKF) \cite{HLT16}. For the latter there has been extensive work on improving the theory \cite{LTT16}, extending the original MLEnKF \cite{CHL20,HST20}, which includes the infinite-dimensional setting, and with applications in data assimilation procedures \cite{FMS20} such as geophysical models.

The purpose of this article is to extend MLMC procedures for the EnKF where we consider the ensemble Kalman--Bucy filter (EnKBF). This filter
is an $N-$particle approximation of the Kalman-Bucy diffusion process, which itself can be considered as a conditional type McKean-Vlasov diffusion process. The Kalman-Bucy diffusion admits a Gaussian marginal time law whose mean and variance which is identical to that of a linear filter of the type \eqref{eq:dat}-\eqref{eq:sig}. As the Kalman-Bucy diffusion cannot typically be simulated, one uses $N$ samples simulated in parallel and recursively in time, to approximate the law of the process. The convergence of this process has been well studied for instance in \cite{BD20,BD17,DT18,DKT17}.


Our motivation for developing MLMC strategies for the EnKBF, termed the 
multilevel Ensemble Kalman--Bucy filter (MLEnKBF), are four-fold: firstly (i) as mentioned it can be viewed as a continuous-time version of the EnKF, which is of interest as data now comes in a continuous-form, which is apparent in certain applications. Most of the literature on multilevel methods applied to filters has primarily been in the discrete-time case. However some recent work has bee done on understanding MLMC for continuous-time filtering problems \cite{JYH20}. Secondly (ii) the EnKBF has connections with new particle filtering methodologies such as the feedback particle filter \cite{TWM18,WRS18,WT19}, which, for particular setups, has shown to overcome the curse of dimensionality.  Thirdly (iii) despite the connection with the EnKF, the EnKBF is different in the form it takes, which as mentioned inherently makes it more difficult to derive any analysis. 
Lastly (iv) as it assumes a linear and Gaussian setting it coincides with the optimal filter, the KBF, which provides a motivation for linear filtering. 

\subsection{Contributions}
Our contributions of this manuscript are highlighted through the following points:
\begin{itemize}
\item We firstly consider understanding the Monte Carlo approximation and convergence of the EnKBF with respect to the KBF. In the continuous-time setting this has been shown in \cite{DT18}, however our results are specific in a discretized setting where we adopt an Euler discretization. This is crucial in order to proceed to the multilevel case. As a result we are able to show {an almost sure convergence of the estimator to a Gaussian associated to the discretized KBF.}
\item We develop a multilevel version of the EnKBF, which we refer to as the MLEnKBF. We provide a mathematical analysis of the mean square error of the MLEnKBF estimator
of the filter. {Our analysis has the limitation that it shows that the system actually simulated is `close' for a large number of samples to an ideal system. We then directly analyze
that ideal system.}
 Our analysis shows, for the ideal system, that in the multilevel setting to achieve a mean square error of $\mathcal{O}(\epsilon^2)$, for $\epsilon>0$, we require a cost of $\mathcal{O}(\epsilon^{-2}\log(\epsilon)^2)$ which is a reduction in cost compared to the EnKBF which is $\mathcal{O}(\epsilon^{-3})$.
\item {To illustrate our findings we provide numerical experiments that highlight the reduction in cost for the MLEnKBF. This is provided for high dimensional examples of the order $\sim \mathcal{O}(10^4)$ and $\mathcal{O}(10^5)$. Furthermore we also numerically show that a deterministic counterpart to the EnKBF, can benefit from the use of MLMC, which are tested on an Ornstein--Uhlenbeck process}.
\end{itemize}
\subsection{Outline}
This paper is organized as follows: We begin with Section \ref{sec:model} where we introduce our methodology of applying multilevel techniques to the ensemble Kalman--Bucy diffusion. In Section \ref{sec:theory} we present our theoretical result which is an upper-bound on the mean square error of our  MLEnKBF estimator.
In Section \ref{sec:num} {we test our theory on an Ornstein--Uhlenbeck process on a number of high-dimensional examples}. Finally we conclude our findings in Section \ref{sec:conc}. The proof of the main result in Section \ref{sec:theory} is provided in the Appendix.


\section{Model and Method}
\label{sec:model}

In this section we provide a concise overview on Kalman--Bucy filtering, which is specific to the linear Gaussian case. We extend this to the algorithm of interest which is the ensemble Kalman--Bucy filter (EnKBF). After doing so we then introduce the concept of Multilevel Monte Carlo and apply it to the EnKBF. This leads to the proposed method of the multilevel ensemble Kalman--Bucy filter. Below, for any probability measure $\pi$ and $\pi-$integrable and possibly matrix-valued function $\varphi$ we write the expectation
of $\varphi$ w.r.t.~$\pi$ as $\pi(\varphi)$.

\subsection{Kalman--Bucy Filters}

We consider the linear filtering problem, 
\begin{align}
\label{eq:data}
dY_t & =  CX_t dt + R_2^{1/2} dV_t, \\
\label{eq:signal}
dX_t & =  A X_t dt + R_1^{1/2} dW_t,
\end{align}
where $(Y_t,X_t)\in\mathbb{R}^{d_y}\times\mathbb{R}^{d_x}$, $(V_t,W_t)$ is a $(d_y+d_x)-$dimensional standard Brownian motion, 
$A$ is a square $d_x\times d_x$ matrix, $C$ is a $d_y\times d_x$ matrix, $Y_0=0$, $X_0\sim\mathcal{N}_{d_x}(\mathcal{M}_0,\mathcal{P}_0)$
($d_x-$dimensional Gaussian distribution, mean $\mathcal{M}_0$, covariance matrix $\mathcal{P}_0$)
and $R_1^{1/2},R_2^{1/2}$ are square (of the appropriate dimension) and symmetric and invertible matrices. It is well-known that, letting
$\{\mathscr{F}_t\}_{t\geq 0}$ be the filtration generated by the observations, the conditional probability of $X_t$ given $\mathscr{F}_t$ is a
Gaussian distribution with mean and covariance matrix 
$$
\mathcal{M}_t := \mathbb{E}[X_t|\mathscr{F}_t], \quad {\mathcal{P}_t := \mathbb{E}\Big[[X_t - \mathbb{E}(X_t|\mathscr{F}_t)][X_t - \mathbb{E}(X_t|\mathscr{F}_t)]^{\top}\Big]},
$$
given by the Kalman--Bucy and Ricatti equations \cite{DT18}
\begin{align}
\label{eq:kbf}
d\mathcal{M}_t &= A \mathcal{M}_t dt + \mathcal{P}_tC^{\top}R^{-1}_2\Big(dY_t - C{\mathcal{M}_t}dt\Big), \\
\label{eq:ricc}
\partial_t\mathcal{P}_t &= \textrm{Ricc}(\mathcal{P}_t),
\end{align}
where the Riccati drift is defined as
$$
\textrm{Ricc}(Q) = AQ + QA^{\top}-QSQ + R, \quad \textrm{with} \ R =R_1 \quad \textrm{and} \ S:=C^{\top}R^{-1}_2C.
$$
 A derivation of \eqref{eq:kbf} - \eqref{eq:ricc} can be found in \cite{AJ70}

The Kalman--Bucy diffusion is a conditional McKean-Vlasov diffusion process (e.g.~\cite{BD20,BD19b}):
\begin{equation}
\label{eq:non-lin}
d\overline{X}_t = A\overline{X}_t dt + R_1^{1/2}d\overline{W}_t + \mathcal{P}_{\eta_t}C^{\top}R_2^{-1}\Big(dY_t -
\Big[C\overline{X}_tdt+R_2^{1/2}d\overline{V}_t\Big]\Big),
\end{equation}
where $(\overline{V}_t,\overline{W}_t,\overline{X}_0)$ are independent copies of $(V_t,W_t,X_0)$ and covariance
$$
\mathcal{P}_{\eta_t} = \eta_t\Big([e-\eta_t(e)][e-\eta_t(e)]^{\top}\Big), \quad \eta_t:= \mathrm{Law}(\overline{X}_t|\mathscr{F}_t),
$$
such that $\eta_t$ is the conditional law of $\overline{X}_t$ given $\mathscr{F}_t$ and $e(x)=x$. It is important to note that the nonlinearity in \eqref{eq:non-lin} does not depend on the distribution of the state $\mathrm{Law}(\overline{X}_t)$ but on the conditional distribution $\eta_t$, and $\mathcal{P}_{\eta_t}$ alone does not depend on $\mathscr{F}_t$.
It is known that the conditional  expectations of the random states $\overline{X}_t$ and their conditional covariance matrices $\mathcal{P}_{\eta_t}$, w.r.t. $\mathscr{F}_t$, satisfy the Kalman--Bucy and the Riccati equations. In addition, for any $t\in\mathbb{R}^+$
$$
\eta_t:= \mathrm{Law}(\overline{X}_t|\mathscr{F}_t) = \mathrm{Law}({X}_t|\mathscr{F}_t).
$$
As a result, an alternative to recursively computing \eqref{eq:kbf} - \eqref{eq:ricc}, is to generate $N$ i.i.d.~samples from 
\eqref{eq:non-lin} and apply a Monte Carlo approximation, which we now discuss.

\subsection{Ensemble Kalman--Bucy Filter}

Exact simulation from \eqref{eq:non-lin} is typically not possible, as one cannot compute $\mathcal{P}_{\eta_t}$ exactly. 
The ensemble Kalman--Bucy filter (EnKBF) can be used to deal with this issue. The EnKBF coincides with the mean-field particle interpretation of the diffusion \eqref{eq:non-lin}. The EnKBF is an $N-$particle system that is simulated as follows, for the $i^{th}-$particle, $i\in\{1,\dots,N\}$:
\begin{equation}
\label{eq:enkbf}
d\xi_t^i = A\xi_t^i dt + R_1^{1/2} d\overline{W}_t^i + {{U}_t^N}\Big(dY^{{i}}_t -
\Big[C\xi_t^idt+R_2^{1/2}d\overline{V}_t^i\Big]\Big),
\end{equation}
such that ${U_t^N={P}_t^NC^{\top}R_2^{-1}}$ and
\begin{align*}
{P}_t^N & = {\Big(1-\frac{1}{N}\Big)^{-1}\mathcal{P}_{\eta_t^N}}= \frac{1}{N -1}\sum_{i=1}^N (\xi_t^i-m^N_t)(\xi_t^i-m^N_t)^{\top},\\
m^N_t & =  \frac{1}{N }\sum_{i=1}^N \xi_t^i,
\end{align*}
where $\xi_0^i\stackrel{\textrm{i.i.d.}}{\sim}\mathcal{N}_{d_x}(\mathcal{M}_0,\mathcal{P}_0)$. It is remarked that when $C=0$, \eqref{eq:enkbf} reduces to $N$ independent copies of an Ornstein--Uhlenbeck process. The EnKBF, KBF and its properties have been well studied, most notably by Del Moral, de Wiljes and coauthors \cite{DT18,DKT17,WRS18,WT19}. 

%


In practice, one will not have access to an entire trajectory of observations. Thus numerically, one often works with a time discretization, such as the Euler method.
Let $\Delta_l=2^{-l}$ then we will generate the system for $(i,k)\in\{1,\dots,N\}\times\mathbb{N}_0=\mathbb{N}\cup\{0\}$ as
\begin{align}
\label{eq:enkf_ps}
\xi_{(k+1)\Delta_l}^i &= \xi_{k\Delta_l}^i + A\xi_{k\Delta_l}^i\Delta_l + R_1^{1/2} [\overline{W}_{(k+1)\Delta_l}^i-\overline{W}_{k\Delta_l}^i] \\
&+{U_{k\Delta_l}^N}\Big([Y^{{i}}_{(k+1)\Delta_l}-Y^{{i}}_{k\Delta_l}]-
\Big[C\xi_{k\Delta_l}^i\Delta_l+R_2^{1/2}[\overline{V}_{(k+1)\Delta_l}^i-\overline{V}_{k\Delta_l}^i]\Big]\Big), \nonumber
\end{align}
such that {$U_{k\Delta_l}^N=P_{k\Delta_l}^NC^{\top}R_2^{-1}$}
\begin{align*}
P_{k\Delta_l}^N & =  \frac{1}{N -1}\sum_{i=1}^N (\xi_{k\Delta_l}^i-m_{k\Delta_l^N})(\xi_{k\Delta_l}^i-m_{k\Delta_l}^N)^{\top}, \\
m_{k\Delta_l}^N & =  \frac{1}{N }\sum_{i=1}^N \xi_{k\Delta_l}^i,
\end{align*}
and $\xi_0^i\stackrel{\textrm{i.i.d.}}{\sim}\mathcal{N}_{d_x}(\mathcal{M}_0,\mathcal{P}_0)$. For $l\in\mathbb{N}_0$ given, denote by $\eta_t^{N,l}$ as the $N-$empirical
probability measure of the particles $(\xi_t^1,\dots,\xi_t^N)$, where $t\in\{0,\Delta_l,2\Delta_l,\dots\}$. For $\varphi:\mathbb{R}^{d_x}\rightarrow\mathbb{R}^{d_x}$
we will use the notation $\eta_t^{N,l}(\varphi):=\tfrac{1}{N}\sum_{i=1}^N\varphi(\xi_t^{i})$. 

\subsection{Convergence}\label{sec:conv_enkbf}

The convergence of this discretized particle system is established below. {To assist our exposition, one can consider the i.i.d.~particle system based upon the Euler discretization of the EnKBF \eqref{eq:enkbf},  for $(i,k)\in\{1,\dots,N\}\times\mathbb{N}_0$
\begin{align}\label{eq:enkf_is}
\zeta_{(k+1)\Delta_l}^i &=(I+A\Delta_l)\zeta_{k\Delta_l}^i + R_1^{1/2} [\overline{W}_{(k+1)\Delta_l}^i-\overline{W}_{k\Delta_l}^i]  \\
&+ U_{k\Delta_l}\Big([Y_{(k+1)\Delta_l}-Y_{k\Delta_l}] -
\Big[C\zeta_{k\Delta_l}^i\Delta_l+R_2^{1/2}[\overline{V}_{(k+1)\Delta_l}^i-\overline{V}_{k\Delta_l}^i]\Big]\Big), \nonumber
\end{align}
where $\zeta_0^i\stackrel{\textrm{i.i.d.}}{\sim}\mathcal{N}_{d_x}(\mathcal{M}_0,\mathcal{P}_0)$ and $U_{k\Delta_l}=P_{k\Delta_l} C^{\top}R_2^{-1}$. It is straightforward to show, using a characteristic function argument, that for any $(i,k)\in\{1,\dots,N\}\times\mathbb{N}_0$, $\zeta_{(k+1)\Delta_l}^i|\mathscr{F}_{(k+1)\Delta_l}\stackrel{\textrm{i.i.d.}}{\sim}\mathcal{N}_{d_x}(m_{(k+1)\Delta_l},P_{(k+1)\Delta_l})$
(we denote the associated Gaussian probability measure as $\eta_{(k+1)\Delta_l}^l$) and
\begin{align}
m_{(k+1)\Delta_l} & =  m_{k\Delta_l} + Am_{k\Delta_l}\Delta_l + U_{k\Delta_l}\Big(
[Y_{(k+1)\Delta_l}-Y_{k\Delta_l}] -Cm_{k\Delta_l}\Delta_l
\Big),\label{eq:iid_mean_rec}\\
P_{(k+1)\Delta_l} & =  P_{k\Delta_l} + \textrm{Ricc}(P_{k\Delta_l})\Delta_l +  (A-P_{k\Delta_l}S)P_{k\Delta_l}(A^{\top}-SP_{k\Delta_l})\Delta_l^2.\label{eq:iid_cov_rec}
\end{align}
One can note that the recursion of the mean is a first-order discretizationof the Kalman-Bucy filter and the recursion for the covariance is a type of second order discretization of the Riccati equation.}

Recall $e(x)=x$.
For a $d_x-$dimensional vector $x$ denote $\|x\|_2=(\sum_{j=1}^{d_x}x(j)^2)^{1/2}$, where $x(j)$ is the $j^{th}-$element of $x$. 
We use the notation $[\eta_{t+k_1\Delta_l}^{N,l}-\eta_{t+k_1\Delta_l}^l](e)=\eta_{t+k_1\Delta_l}^{N,l}(e)-\eta_{t+k_1\Delta_l}^l(e)$.
\begin{proposition}\label{prop:var_term1_sec_state}
{For any $(l,t,k_1)\in\mathbb{N}_0\times\mathbb{R}^+\{0,1,\dots, \Delta_l^{-1}\}$ almost surely:
$$
 \lim_{N\rightarrow\infty}[\eta_{t+k_1\Delta_l}^{N,l}-\eta_{t+k_1\Delta_l}^l](e) = 0.
$$}
\end{proposition}
\begin{proof}
{The proof follows by using Theorem \ref{theo:prop} in the appendix and the Marcinkiewicz-Zygmund inequality for i.i.d.~random variables, along with a standard first Borel Cantelli lemma argument.}
\end{proof}

{Now if we set $\hat{\eta}_{t+k_1\Delta_l}^{N,l}(e)=\frac{1}{N}\sum_{i=1}^N\zeta_{t+k_1\Delta_l}^i$, using the Markov inequality and Theorem \ref{theo:prop} one can show that for any $\varepsilon>0$ and $q>0$:
\begin{equation}\label{eq:prob_est}
\mathbb{P}\left(\left|\eta_{t+k_1\Delta_l}^{N,l}(e)-\hat{\eta}_{t+k_1\Delta_l}^{N,l}(e)\right|>\varepsilon\right) \leq \frac{\mathsf{C}}{\varepsilon^{2q} N^{q/2}}.
\end{equation}
where $\mathsf{C}$ is a constant that can depend on $(l,q,t,k_1)$ but not $N$. This implies that for $N$ large, the system described by the recursion \eqref{eq:enkf_ps} shares
the properties of the system \eqref{eq:enkf_is}. As a result, we now simply consider the latter system in our subsequent discussion.
Note that the following analysis does not consider the time parameter {$t$} and this is discussed in Section \ref{sec:theory}.
In Proposition \ref{prop:var_term1}, in the appendix, we have shown that:
$$
\mathbb{E}\Big[\Big\|[\hat{\eta}_{t+k_1\Delta_l}^{N,l}-\eta_{t+k_1\Delta_l}](e)\Big\|_2^2\Big] \leq \mathsf{C}\Big(\frac{1}{N}+\Delta_l^2\Big),
$$
where $\mathsf{C}$ does not depend upon $l$ nor $N$ but can depend on $t$. We shall seek to control this latter MSE.
Let $\epsilon>0$, to achieve a MSE of 
$\mathcal{O}(\epsilon^2)$  one must set $N=\mathcal{O}(\epsilon^{-2})$
 and $l=\lfloor|\log(\epsilon)|/\log(2)\rfloor+\mathcal{O}(1)$. The cost is $\mathcal{O}(N\Delta_l^{-1})$ to achieve a $\mathcal{O}(\epsilon^2)$ MSE, which is $\mathcal{O}(\epsilon^{-3})$ under our choice of $l$ and $N$.}

\subsection{Multilevel Monte Carlo}

Before introducing our proposed multilevel EnKBF, we first provide a brief overview on the idea of the MLMC \cite{MBG08,MBG15,hein}. 
Let $\pi$ be a probability on a measurable space $(\mathsf{X},\mathscr{X})$ and for $\pi-$integrable $\varphi:\mathsf{X}\rightarrow\mathbb{R}$
consider the problem of estimating $\pi(\varphi)=\mathbb{E}_{\pi}[\varphi(X)]$. We assume that we only have access to a sequence of approximations of $\pi$,  $\{\pi_l\}_{l\in\mathbb{N}_0}$, also
each defined on $(\mathsf{X},\mathscr{X})$ and we are now interested in estimating $\pi_l(\varphi)$, such that 
$\lim_{l\rightarrow\infty}|[\pi_l-\pi](\varphi)|=0$. 
Note that it is assumed that (for instance) the cost of simulation from $\pi_l$ increases with $l$, but, the approximation error between $\pi$ and $\pi_l$ is also falling as $l$ grows.
Then one can consider the telescoping sum
$$
\pi_L(\varphi) = \pi_0(\varphi) + \sum^L_{l=1}[\pi_l-\pi_{l-1}](\varphi).
$$
Then the idea is as follows:
\smallskip
\begin{enumerate}
\item{Approximate $\pi_0(\varphi)$ by using i.i.d.~sampling from $\pi_0$.}
\item{Independently for each $l\in\{1,\dots,L\}$ and the sampling in 1.~approximate $[\pi_l-\pi_{l-1}](\varphi)$ by i.i.d.~sampling from a coupling of $(\pi_l,\pi_{l-1})$.}
\end{enumerate}
The key point is to construct a coupling in 2.~so that the mean square error (MSE) can be reduced relative to i.i.d.~sampling from $\pi_L$. This latter construction often
relies on the specific properties of $(\pi_l,\pi_{l-1})$. 

Denoting $N_0\in\mathbb{N}$ i.i.d.~samples from $\pi_0$ as $(X^{1,0},\dots,X^{N_0,0})$ and for $l\in\{1,\dots,L\}$, $N_l\in\mathbb{N}$ samples
from a coupling of $(\pi_l,\pi_{l-1})$ as $((X^{1,l},\tilde{X}^{1,l-1}),\dots,(X^{N_l,l},\tilde{X}^{N_l,l-1}))$, one has the MLMC approximation of $\mathbb{E}_{\pi_L}[\varphi(X)]$
$$
\pi_L^{ML}(\varphi) := \frac{1}{N_0}\sum_{i=1}^{N_0}\varphi(X^{i,0}) + \sum_{l=1}^L \frac{1}{N_l}\sum_{i=1}^{N_l}\{\varphi(X^{i,l})-\varphi(\tilde{X}^{i,l-1})\}.
$$
The MSE is then
$$
\mathbb{E}[(\pi_L^{ML}(\varphi)-\pi(\varphi))^2] = \mathbb{V}\textrm{ar}[\pi_L^{ML}(\varphi)] + 
[\pi_L-\pi](\varphi)^2,
$$
where $\mathbb{V}\textrm{ar}[\cdot]$ denotes the variance (which we assume exists). One has
$$
\mathbb{V}\textrm{ar}[\pi_L^{ML}(\varphi)] = \Bigg(\frac{\mathbb{V}\textrm{ar}[\varphi(X^{1,0})]}{N_0}+\sum_{l=1}^L\frac{
\mathbb{V}\textrm{ar}[\varphi(X^{1,l})-\varphi(\tilde{X}^{1,l-1})]
}{N_l}\Bigg).
$$
Thus if {$\mathbb{V}\textrm{ar}[\varphi(X^{1,l})-\varphi(\tilde{X}^{1,l-1})]$ falls sufficiently fast with $l$, and given an appropriate characterization of: the bias $
[\pi_L-\pi](\varphi)$,  and the cost as a function of $l$ it is possible to choose $N_l$ and $L$ to improve upon the i.i.d.~estimator
$$
\frac{1}{N}\sum_{i=1}^N \varphi(X^{i,L}),
$$
where $X^{i,L}$ are i.i.d.~from $\pi_L$. MLMC has been particularly popular when dealing with stochastic models, such as stochastic differential equations, but recently has been applied to filtering in particular the discrete-time EnKF \cite{CHL20,HLT16,HST20}. 



\subsection{Multilevel EnKBF}

To enhance the efficiency of using just \eqref{eq:enkf_ps}, we consider a coupled Ensemble Kalman--Bucy filter in a similar way to \cite{HLT16}. Let $l\in\mathbb{N}$ then  
we run the coupled system for $(i,k)\in\{1,\dots,N\}\times\mathbb{N}_0$
\begin{align}
\xi_{(k+1)\Delta_l}^{i,l} & = \xi_{k\Delta_l}^{i,l} + A\xi_{k\Delta_l}^{i,l}\Delta_l + R_1^{1/2} [\overline{W}_{(k+1)\Delta_l}^i-\overline{W}_{k\Delta_l}^i]  
+ U_{k\Delta_l}^{N,l}\Big([Y^{{i}}_{(k+1)\Delta_l}-Y^{{i}}_{k\Delta_l}]\nonumber \\ &-\Big[C\xi_{k\Delta_l}^{i,l}\Delta_l + R_2^{1/2}[\overline{V}_{(k+1)\Delta_l}^i-\overline{V}_{k\Delta_l}^i]\Big]\Big),  \label{eq:ml_enkbf1}\\
\xi_{(k+1)\Delta_{l-1}}^{i,l-1} & =  \xi_{k\Delta_{l-1}}^{i,l-1} + A\xi_{k\Delta_{l-1}}^{i,l-1}\Delta_{l-1} + R_1^{1/2} [\overline{W}_{(k+1)\Delta_{l-1}}^i-\overline{W}_{k\Delta_{l-1}}^i]+ U_{k\Delta_{l-1}}^{N,l-1}\Big([Y^{{i}}_{(k+1)\Delta_{l-1}}-Y^{{i}}_{k\Delta_{l-1}}] \nonumber \\&-\Big[C\xi_{k\Delta_{l-1}}^{i,l-1}\Delta_{l-1}+ R_2^{1/2}[\overline{V}_{(k+1)\Delta_{l-1}}^i-\overline{V}_{k\Delta_{l-1}}^i]\Big]\Big),\label{eq:ml_enkbf2}
\end{align}
where $U_{k\Delta_s}^{N,s}=P_{k\Delta_s}^{N,s}C^{\top}R_2^{-1}$, $s\in\{l-1,l\}$, and our sample covariances and means are defined as
\begin{align*}
P_{k\Delta_l}^{N,l} & =  \frac{1}{N -1}\sum_{i=1}^N (\xi_{k\Delta_l}^{i,l}-m_{k\Delta_l}^{N,l})(\xi_{k\Delta_l}^{i,l}-m_{k\Delta_l}^{N,l})^{\top} ,\\
m_{k\Delta_l}^{N,l} & =  \frac{1}{N }\sum_{i=1}^N \xi_{k\Delta_l}^{i,l},\\
P_{k\Delta_{l-1}}^{N,l-1} & =  \frac{1}{N -1}\sum_{i=1}^N (\xi_{k\Delta_{l-1}}^{i,l-1}-m_{k\Delta_{l-1}}^{N,l-1})(\xi_{k\Delta_{l-1}}^{i,l-1}-m_{k\Delta_{l-1}}^{N,l-1})^{\top}, \\
m_{k\Delta_{l-1}}^{N,l-1} & =  \frac{1}{N }\sum_{i=1}^N \xi_{k\Delta_{l-1}}^{i,l-1},
\end{align*}
and $\xi_0^{i,l}\stackrel{\textrm{i.i.d.}}{\sim}\mathcal{N}_{d_x}(\mathcal{M}_0,\mathcal{P}_0)$, $\xi_{0}^{i,l-1}=\xi_0^{i,l}$.
Then, one has the approximation of $[\eta_t^l-\eta_t^{l-1}](\varphi)$, $t\in\mathbb{N}_0$, $\varphi:\mathbb{R}^{d_x}\rightarrow\mathbb{R}$, given by
$$
[\eta_t^{N,l} -\eta_t^{N,l-1}](\varphi) = \frac{1}{N}\sum_{i=1}^N[\varphi(\xi_t^{i,l})-\varphi(\xi_t^{i,l-1})].
$$
Note that estimation of the filter at any time $t\in\{\Delta_{l-1},2\Delta_{l-1},\dots\}$ can easily be done in the same way as above.

Applying MLMC to our problem, consists of:
\begin{enumerate}
\item{Running the discretized Euler ensemble Kalman--Bucy filter (as in \eqref{eq:enkf_ps}) at level 0 with $N_0$ particles.}
\item{Independently for $l\in\{1,\dots,L\}$ and the sampling in 1., ~run the coupled discretized Euler ensemble Kalman--Bucy filter (as in \eqref{eq:ml_enkbf1}-\eqref{eq:ml_enkbf2}) with $N_l$ particles.}
\end{enumerate}
So one has an approximation for $t\in\mathbb{N}_0$
\begin{equation}\label{eq:main_est}
\eta_t^{ML}(\varphi):=\eta_t^{N_0,0}(\varphi) + \sum_{l-1}^L [\eta_t^{N_l,l} -\eta_t^{N_l,l-1}](\varphi).
\end{equation}
We use $\mathbb{E}[\cdot]$ to denote expectations w.r.t.~the law of the process just described (which naturally includes the data process).
Different types of Kalman-Bucy diffusions and discretizations are considered in Section \ref{sec:num}.


\section{Main Result}
\label{sec:theory}
Having introduced the relevant mathematics and our new methodology, in this section we provide our main theoretical result associated to the estimate \eqref{eq:main_est}.
All of the proofs are omitted and presented in the appendix. $N_{0:L}=(N_0,\dots,N_L)^{\top}$.

{Our strategy, is much the same as in Section \ref{sec:conv_enkbf}. We first consider an i.i.d.~coupled particle system: Let $l\in\mathbb{N}$ be fixed.
To analyze the multilevel d-EnBKF (MLd-EnKBF) we will prove results for the i.i.d.~coupled particle system for $(i,k)\in\{1,\dots,N\}\times\mathbb{N}_0$:
\begin{align}
\zeta_{(k+1)\Delta_l}^{i,l} & = \zeta_{k\Delta_l}^{i,l} + A\zeta_{k\Delta_l}^{i,l}\Delta_l + R_1^{1/2} [\overline{W}_{(k+1)\Delta_l}^i-\overline{W}_{k\Delta_l}^i]  
+ U_{k\Delta_l}^{l}\Big([Y^{{i}}_{(k+1)\Delta_l}-Y^{{i}}_{k\Delta_l}] \label{eq:iid1}\\ &-\Big[C\zeta_{k\Delta_l}^{i,l}\Delta_l + R_2^{1/2}[\overline{V}_{(k+1)\Delta_l}^i-\overline{V}_{k\Delta_l}^i]\Big]\Big), \nonumber \\
\zeta_{(k+1)\Delta_{l-1}}^{i,l-1} & =  \xi_{k\Delta_{l-1}}^{i,l-1} + A\zeta_{k\Delta_{l-1}}^{i,l-1}\Delta_{l-1} + R_1^{1/2} [\overline{W}_{(k+1)\Delta_{l-1}}^i-\overline{W}_{k\Delta_{l-1}}^i] 
+ U_{k\Delta_{l-1}}^{l-1}\Big([Y^{{i}}_{(k+1)\Delta_{l-1}}-Y^{{i}}_{k\Delta_{l-1}}] \label{eq:iid2}\\ &-\Big[C\zeta_{k\Delta_{l-1}}^{i,l-1}\Delta_{l-1}+ R_2^{1/2}[\overline{V}_{(k+1)\Delta_{l-1}}^i-\overline{V}_{k\Delta_{l-1}}^i]\Big]\Big),\nonumber
\end{align}
where $U_{k\Delta_s}^s=P_{k\Delta_s}^{s}C^{\top}R_2^{-1}$, $s\in\{l-1,l\}$, and one should recall \eqref{eq:deter_cov_evol}-\eqref{eq:deter_mean_evol}. Note also $\xi_0^{i,l}=\xi_0^{i,l-1}=\zeta_0^{i,l}=\zeta_0^{i,l-1}\stackrel{\textrm{i.i.d.}}{\sim}\mathcal{N}_{d_x}(\mathcal{M}_0,\mathcal{P}_0)$.
It should be stressed that the increments
$[\overline{W}_{(k+1)\Delta_l}^i-\overline{W}_{k\Delta_l}^i]$, $[\overline{V}_{(k+1)\Delta_l}^i-\overline{V}_{k\Delta_l}^i]$,  $[\overline{W}_{(k+1)\Delta_{l-1}}^i-\overline{W}_{k\Delta_{l-1}}^i]$ and $[\overline{V}_{(k+1)\Delta_{l-1}}^i-\overline{V}_{k\Delta_{l-1}}^i]$ are identical to those used in the actual MLd-EnKBF. Then we set
\begin{equation}\label{eq:main_est_iid}
\hat{\eta}_t^{ML}(\varphi):=\hat{\eta}_t^{N_0,0}(\varphi) + \sum_{l-1}^L [\hat{\eta}_t^{N_l,l} -\hat{\eta}_t^{N_l,l-1}](\varphi),
\end{equation}
where $[\hat{\eta}_t^{N_l,l} -\hat{\eta}_t^{N_l,l-1}](\varphi)=\frac{1}{N}\sum_{i=1}^N[\varphi(\zeta_{t}^{i,l})-\varphi(\zeta_{t}^{i,l-1})]$.}

{Now following the analysis associated to \eqref{eq:prob_est} one can not only establish that, almost surely,
$$
\lim_{\min_l N_l\rightarrow\infty}[\eta_t^{ML}-\hat{\eta}_t^{ML}](e) = 0,
$$
but that for any $\varepsilon>0$ and $q>0$:
$$
\mathbb{P}\left(\left|[\eta_t^{ML}-\hat{\eta}_t^{ML}](e)\right|>\varepsilon\right) \leq \frac{\mathsf{C}}{\varepsilon^{2q}}\left(\sum_{l=0}^L \frac{1}{N_l^{q/2}}\right),
$$
where $\mathsf{C}$ is a constant that can depend on $(L,q,t)$ but not $N_{0:L}$. So again, if $N_{0:L}$ are sufficiently large
the estimator that is actually used $\eta_t^{ML}(e)$ will share the properties of the estimator $\hat{\eta}_t^{ML}(e)$ and we simply consider the latter.
}

\begin{theorem}\label{theo:main_theo}
For any $T\in\mathbb{N}$ fixed and $t\in[0,T-1]$ there exists a $\mathsf{C}<+\infty$ such that for any $(L,N_{0:L})\in\mathbb{N}\times\{2,3,\dots\}^{L+1}$,
$$
\mathbb{E}\left[\left\|[\hat{\eta}_t^{ML}-\eta_t](e)\right\|_2^2\right] \leq \mathsf{C}\left(
\sum_{l=0}^L \frac{\Delta_l}{N_l} + \sum_{l=1}^L\sum_{q=1, q\neq l}^L\frac{\Delta_l\Delta_q}{N_lN_q} + \Delta_L^2
\right).
$$
\end{theorem}

Theorem \ref{theo:main_theo} is proved for $t\in\{0,1,\dots,T\}$, however, if one considered the multilevel identity with the largest discretization $\Delta_m$, then
the result can easily be extended to the case $t\in\{0,\Delta_m,2\Delta_m,\dots,T\}$. The frequency of estimation is simply a by-product of considering the
discretization levels $\Delta_0,\dots,\Delta_L$, but this need not be the case in practice.

The subsequent discussion ignores the time parameter {$t$}; this is considered below.
The main implication of this result is as follows (and as considered in e.g.~\cite{MBG08}). Let $\epsilon>0$ be given, if one chooses {$L=\lfloor|\log(\epsilon)|/\log(2)\rfloor+\mathcal{O}(1)$}
and $N_l=\mathcal{O}(\epsilon^{-2}\Delta_l|\log(\epsilon)|)$ then the upper-bound in Theorem \ref{theo:main_theo} indicates that the mean square error is $\mathcal{O}(\epsilon^2)$.
The cost to achieve this is $\mathcal{O}(\sum_{l=0}^L N_l\Delta_l^{-1})$, which is $\mathcal{O}(\epsilon^{-2}\log(\epsilon)^2)$ under our choices of $L$ and $N_{0:L}$.
As noted previously, if one ran a single EnKBF then the cost to achieve the same mean square error is $\mathcal{O}(\epsilon^{-3})$. 
 
The results related to this section have several assumptions, which we discuss below through the following remarks.

\begin{remark}
\label{rem:1}
The first is the dependence of the constant $\mathsf{C}$ on $t, T$, which is typically exponential in $T$.
This exponential dependence could be dealt with by trying to use Foster-Lyapunov/martingale methods such as in \cite{DT18,TMK18}. In our context, we did not find an appropriate Lyapunov function and it seems intrinsically more challenging than for the discrete-time EnKF considered in \cite{TMK18}.  In connection to this, we have not assumed that the signal process is stable (e.g.~\cite{DT18}) which could possibly circumvent such time uniform bounds. In addition, as noted in \cite{BD20} time uniform errors under Euler-discretization may not be possible. {We note also that using such Foster-Lyapunov/martingale methods, one might also be able to analyze $\eta_t^{ML}(e)$ directly. Our analysis is based upon a time recursion, which is not amenable to the required uniform in time and level bounds.}
\end{remark}

\section{Numerical Examples}
\label{sec:num}
In this section we aim to verify the rates attained with respect to the MSE of the MLEnKBF. We will compare the accuracy and the computational cost, of the performance, to the EnKBF. This will be tested firstly on a {high dimensional example, of the order $10^4$, where after we increase the order of dimension to $10^5$}.  This is to motivate practical applications such as numerical weather prediction which operates with moderate to high-dimensional models. Our numerics will be conducted on an Ornstein--Uhlenbeck process, where we test both the MLEnKBF and a deterministic variant, which we introduce below.

\subsection{Deterministic EnKBF}

{A variant of the described EnKBF is the deterministic version, which contains no perturbed observations, was first introduced in \cite{BR12}, motivated from the EnKF \cite{SO08}. This modifies \eqref{eq:enkbf} to 
\begin{equation}
\label{eq:enkbf2}
d\xi_t^i = A\xi_t^i dt + R_1^{1/2} d\overline{W}_t^i + U_t^N\Bigg(dY_t -
\frac{C\xi_t^i+Cm_t}{2}dt\Bigg),
\end{equation}
where {$U_t^N=P_t^NC^{\top}R_2^{-1}$} and we exclude the additive noise from \eqref{eq:data} and replace it with a $Cm_tdt$ term, which includes the sample mean. It also known that in the linear setting, as $N \rightarrow \infty$, that \eqref{eq:enkbf2} is consistent with the moments of the vanilla KF. {The motivation for numerically studying \eqref{eq:enkbf2} is that the deterministic variant is a special case of the feedback particle filter  \cite{SKP19,TWM18,YMM13}. This recently developed particle filter has demonstrated promise, as it has been shown not to suffer from the curse of dimensionality, which occurs with the vanilla particle filter.  Also, in general, deterministic ensemble methods alleviate instability issues, and improve on accuracy as demonstrated through filters such as square-root filters \cite{SO08}.}}

We proceed in a similar fashion to derive a deterministic ML(D)EnKBF. We consider the Euler-discretized DEnKBF. Let $\Delta_l=2^{-l}$ then we will generate the system, similarly as before, for $(i,k)\in\{1,\dots,N\}\times\mathbb{N}_0$
\begin{align}
\label{eq:enkf_ps2}
\xi_{(k+1)\Delta_l}^i &= \xi_{k\Delta_l}^i + A\xi_{k\Delta_l}^i\Delta_l + R_1^{1/2} [\overline{W}_{(k+1)\Delta_l}^i-\overline{W}_{k\Delta_l}^i] 
+ U_{k\Delta_l}^N\Bigg([Y_{(k+1)\Delta_l}-Y_{k\Delta_l}] \\ & -
\Bigg[\frac{(C\xi_{k\Delta_l}^i\Delta_l+Cm_{k\Delta_l}\Delta_l)}{2}\Bigg]\Bigg), \nonumber
\end{align}
where {$U_{k\Delta_l}^N=P_{k\Delta_l}^NC^{\top}R_2^{-1}$} and using the same notation and with the mean and covariance defined in the same way as \eqref{eq:enkf_ps}. Now we can consider the coupled system given as
\begin{align*}
\xi_{(k+1)\Delta_l}^{i,l} &= \xi_{k\Delta_l}^{i,l} + A\xi_{k\Delta_l}^{i,l}\Delta_l + R_1^{1/2} [\overline{W}_{(k+1)\Delta_l}^i-\overline{W}_{k\Delta_l}^i] 
+ U_{k\Delta_l}^{N,l}\Bigg([Y_{(k+1)\Delta_l}-Y_{k\Delta_l}] \\&{-}
\Bigg[\frac{(C\xi_{k\Delta_l}^{i,l}\Delta_l+m^{l}_{k\Delta_l}\Delta_l)}{2}\Bigg]\Bigg),  \\
\xi_{(k+1)\Delta_{l-1}}^{i,l-1} &= \xi_{k\Delta_{l-1}}^{i,l-1} + A\xi^{l-1}_{k\Delta_{l-1}}\Delta_{l-1} + R_1^{1/2} [\overline{W}_{(k+1)\Delta_{l-1}}^i-\overline{W}_{k\Delta_{l-1}}^i] 
+ U_{k\Delta_{l-1}}^{N,l-1}\Bigg([Y_{(k+1)\Delta_{l-1}}-Y_{k\Delta_{l-1}}] \\& {-}
\Bigg[\frac{(C\xi_{k\Delta_{l-1}}^i\Delta_{l-1}+m^{l-1}_{k\Delta_{l-1}}\Delta_{l-1})}{2}\Bigg]\Bigg),
\end{align*}
where {$U_{k\Delta_s}^{N,s}=P_{k\Delta_s}^{N,s}C^{\top}R_2^{-1}$, $s\in\{l-1,l\}$.}

{
\begin{remark}
It is important to note that our results, i.e. Proposition \ref{prop:var_term1_sec_state} and Theorem \ref{theo:main_theo}, have not been proved for the DEnKBF and MLDEnKBF. Despite this we expect that there are no significant further challenges in proving such analogous results as the innovation process is the only part of the algorithm which is different. We omit such a proof due to any prolonging of the paper, and leave this for future work. 
\end{remark}
}
\subsection{Numerical Results}
{
 Recall that our signal dynamics is given in the linear form 
$$
dX_t  =  A X_t dt + R_1^{1/2} dW_t.
$$
For the Ornstein--Uhlenbeck (OU) process we take normally randomly generated values for the matrices $A, C, R_1, R_2$, and the initial condition is chosen as ${X_0\sim\mathcal{N}_{d_x}(6,I_d)}$. {Our numerical examples will be considered for firstly a $d_x=10^4$ and $d_y=10^4$ signal and observational process. For the higher dimensional example we choose $d_x=10^5$ and $d_y=10^5$ respectively.  We will test this on the MLEnKBF and the deterministic version, i.e. the MLDEnKBF. For the implementation we take levels {$l\in\{9,10,11,12,13,14\}$} with mesh defined as $\Delta =2^{-l}$, with an Euler--Maruyama discretization for our $10^4$ experiment. For the $10^5$ experiment we proceed similarly but use levels {$l\in\{10,11,12,13,14,15\}$}. Our specific choice for such high levels, is that we can exploit this using HPC. We have omitted using lower levels as for certain levels, i.e. $l=\{0,1,2,3\}$ our numerics indicate little difference in the cost, and it is known that in some situations the EnKBF can produce unstable solutions, using an Euler discretization \cite{BD20}. \\
This phenomena occured for lower levels. We also specify a final time of $T=100$. {Our comparison will be based on the rate for the log cost vs. the log MSE, where the cost of the MLEnKBF is given as $\sum_{l=0}^L N_l\Delta_l^{-1}$, while the cost for the EnKBF it is $N\Delta_l^{-1}$}. The values of $N_l$ and $L$ are chosen as described in Section \ref{sec:theory}. Our numerical experiments were tested on KAUST's supercomputer Shaheen. The Code is written in Python and can be downloaded from \href{https://github.com/fangyuan-ksgk/Multilevel-Ensemble-Kalman-Bucy-Filter}{\textbf{{https://https://github.com/fangyuan-ksgk/Multilevel-Ensemble-Kalman-Bucy-Filter}}}.}}

Our numerical findings are provided in Figures \ref{fig:100d-100d} and \ref{fig:1000d-1000d} which show the effect of MLMC on both the EnKBF and the DEnKBF. The circles and crosses denote the different levels. As we can see from the associated MSEs at level $l=9$, they are relatively similar in Figure \ref{fig:100d-100d}, and similarly for level $l=10$ in Figure \ref{fig:1000d-1000d}.  However once we refine the mesh towards finer levels  we notice a difference in computational cost, where the MLEnKBF methods outperform their counterparts. As expected we see that the rate changes which is verified from the theory of  MLEnKBF, which we recall suggests a reduction of $\mathcal{O}(\epsilon^{-3})$ to $\mathcal{O}(\epsilon^{-2}\log(\epsilon)^2)$. Interestingly when we analyze the deterministic counterpart, {we see an almost identical decay rate of the MSE w.r.t. to the cost, which could suggest similar rates for the MLDEnKBF, as mentioned in Theorem \ref{theo:main_theo}. As our analysis does not cover the later, we leave this as potential future work.}}





\begin{figure}[!htb]
	\begin{subfigure}[c]{0.49\textwidth}
	\includegraphics[width=1.\textwidth]{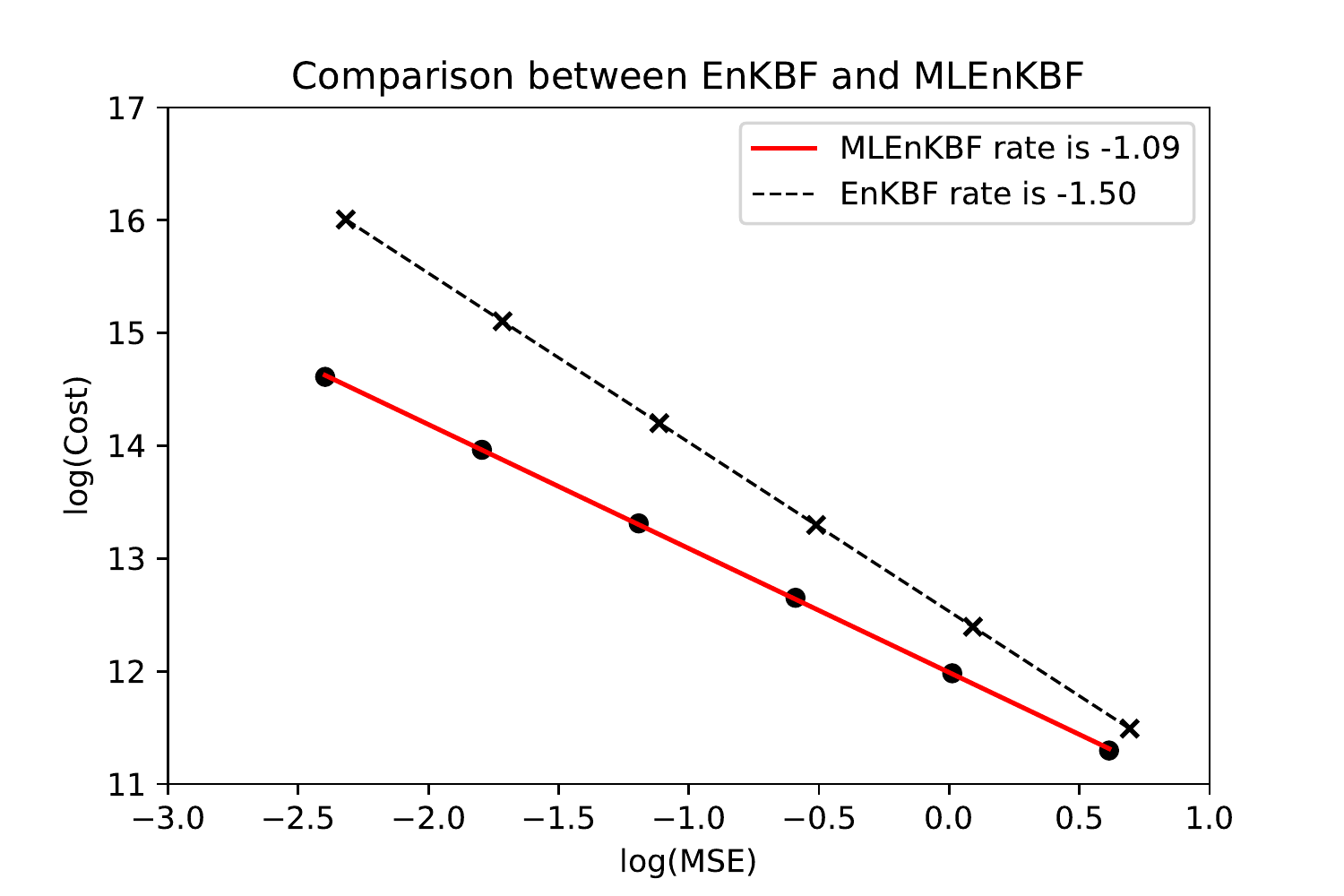}
	\end{subfigure}
	\begin{subfigure}[c]{0.5\textwidth}
	\includegraphics[width=1\textwidth]{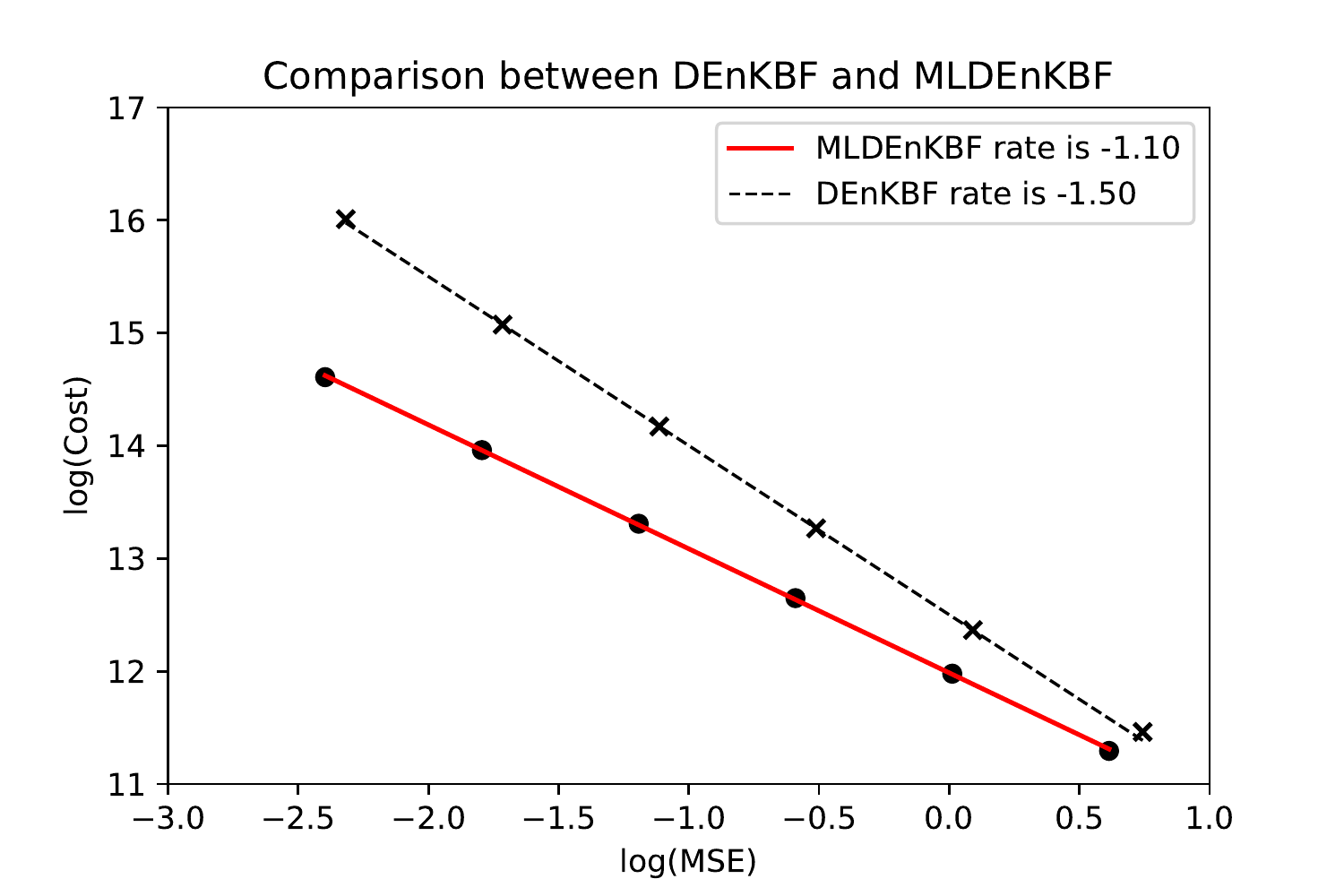}
	\end{subfigure}
    \caption{{Comparison of the multilevel EnKBF algorithms vs. the EnKBF algorithms for a dimension size of $d_x=10^4$ and $d_y=10^4$.}}
    \label{fig:100d-100d}
\end{figure} 

\begin{figure}[!htb]
	\begin{subfigure}[c]{0.49\textwidth}
	\includegraphics[width=1.\textwidth]{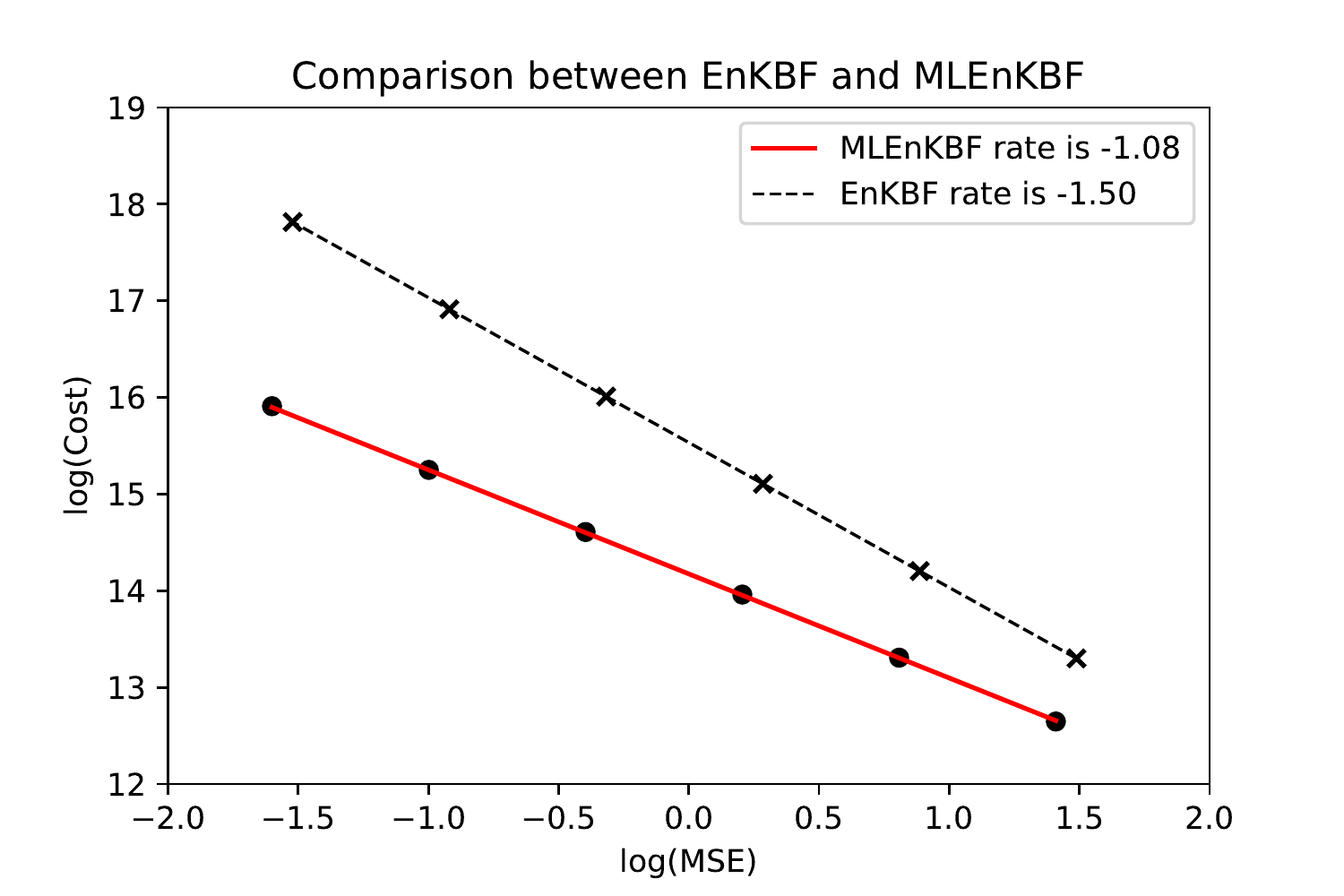}
	\end{subfigure}
	\begin{subfigure}[c]{0.5\textwidth}
	\includegraphics[width=1\textwidth]{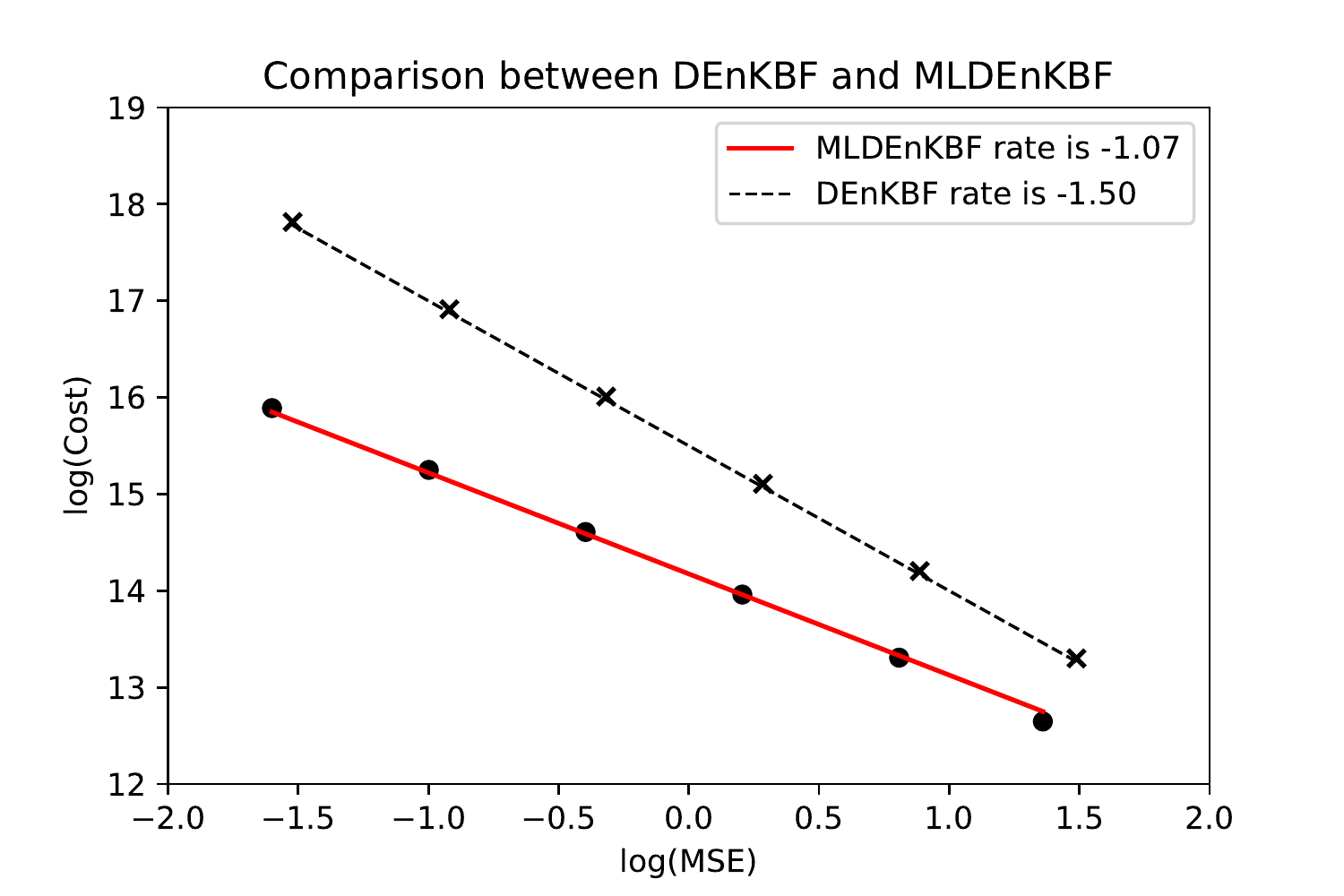}
	\end{subfigure}
    \caption{{Comparison of the multilevel EnKBF algorithms vs. the EnKBF algorithms for a dimension size of $d_x=10^5$ and $d_y=10^5$.}}
    \label{fig:1000d-1000d}
\end{figure}

\section{Conclusion}
\label{sec:conc}
Multilevel strategies for uncertainty quantification have become a very powerful tool at attaining a certain order of mean square error whilst reducing the computational complexity. In this article we introduced and developed an {`ideal'} multilevel methodology for the ensemble Kalman--Bucy filter, which is a continuous-time process filtering method, which coincides with the optimal filter in the linear and Gaussian setting. From our results we presented a Monte-Carlo approximation and convergence for the discretization of the EnKBF, which includes non-uniform, in time, propagation of chaos results. This was then considered for the multilevel case, where we also analyzed the MSE which is related through the statistical and deterministic bias. Our theory was verified numerical experiments taken at dimensions of {order $\sim10^5$}, where we compared the MLEnKBF to its deterministic counterpart.

As this paper is a first step at aiming to provide multilevel strategies for continuous-time Kalman filtering, this naturally leads to a number of interesting directions to consider. Firstly, as mentioned above, we numerically tested our MLEnKF of interest, as well as its deterministic counterpart. However we alleviated any further analysis due to potential prolonging the length of the paper, although we believe the results should also apply similarly. This could also be done for the optimal-transport inspired EnKBF, which is discussed further in \cite{BD20}.  Secondly one could also consider applying localization techniques for multilevel strategies. This has been shown for the discrete MLEnKF \cite{FMS20} and the EnKBF \cite{WT19}, so this would be a natural extension. Another potential direction to consider an is alternative discretization of the diffusion related to the data \eqref{eq:data} up to its finest level, however it is not entirely obvious how one should approach this. Finally as mentioned in Remark \ref{rem:1} it would be of interest to attain uniform (in time) estimates from the MLEnKBF. {Lastly it would be of interest to produce a ML estimator of the EnKBF, which is uniform w.r.t. to the discretization level $l$. As mentioned conventional methods do not naturally apply here, and would provide a fuller understanding of this.}

\section*{Acknowledgments}
This work was supported by KAUST baseline funding.

\appendix

\section{Introduction}
\label{sec:app}

The appendix features several results which are needed to prove Theorem \ref{theo:main_theo}.
This appendix is broken into three sections and is designed to be read in order. It will not be straight-forward to read this appendix out of order.
The first Appendix \ref{app:denknf} features a collection of results associated to the discretized Ensemble Kalman-Bucy Filter (d-EnKBF). {Note 
this is not to be confused with the deterministic Ensemble Kalman-Bucy Filter (DEnKBF) which we used in Section \ref{sec:num}.}
In particular we prove a non-uniform propagation of chaos result at the end of the section.
The second Appendix \ref{app:en_euler} {focuses} upon establishing a strong error result for the Euler discretized Kalman-Bucy diffusion, critical to our final result.
{These type of errors are in the analysis of the i.i.d.~system.}
The final Appendix \ref{app:ml_res} provides results for the i.i.d.~multilevel estimator, concluding with the proof of Theorem \ref{theo:main_theo}.

\section{Results for the Discretized Ensemble Kalman-Bucy Filter}\label{app:denknf}

Throughout the section $l\in\mathbb{N}_0$ is fixed. To shorten the subsequent notations
we write $\textrm{SRicc}(P) = (A-PS)P(A^{\top}-SP)$. For a square matrix $A$, $\textrm{Sym}(A) := \frac{1}{2}(A+A^{\top})$.
We present the following result, intended to parallel \cite[Theorem 3.1]{DT18}, for the d-EnKBF. {The recursion is central to our analysis, of this section,
which seeks to establish the type of convergence in Proposition \ref{prop:var_term1_sec_state} and the property in \eqref{eq:prob_est}, which relies on the afore-mentioned propagation of chaos.}

\begin{proposition}\label{prop:denkf_rec}
For any $(k,l,N)\in\mathbb{N}_0\times\mathbb{N}_0\times\{2,3,\dots\}$ we have:\\\\
(i) Mean Recursion:
\begin{equation}\label{eq:enkf_mean_rec}
m_{(k+1)\Delta_l}^N - m_{k\Delta_l}^N   =  \Big(A - P_{k\Delta_l}^NS\Big)\Delta_l m_{k\Delta_l}^N + U_{k\Delta_l}^N[Y_{(k+1)\Delta_l}-Y_{k\Delta_l}] + \alpha_{k\Delta_l}^N\frac{Z_k}{\sqrt{N}},
\end{equation}
where $Z_k = \frac{1}{\sqrt{N}}\sum_{i=1}^N \omega_k^i$, independently of all other random variables for each $(i,k)\in\{1,\dots,N\}\times \mathbb{N}_0$ $\omega_k^i \stackrel{\textrm{i.i.d.}}{\sim}\mathcal{N}_{d_x}(0,I)$ and $\alpha_{k\Delta_l}^N=(R_1+P_{k\Delta_l}^NSP_{k\Delta_l}^N)^{1/2}\Delta_l^{1/2}$.\\\\
(ii) Covariance Recursion:
\begin{align}
\label{eq:enkf_cov_rec}
P_{(k+1)\Delta_l}^N - P_{k\Delta_l}^N & =\textrm{\emph{Ricc}}(P_{k\Delta_l}^N)\Delta_l + \textrm{\emph{SRicc}}(P_{k\Delta_l}^N)\Delta_l^2
+\alpha_{k\Delta_l}^N \Big(\tfrac{1}{N-1}\sum_{i=1}^N(\omega_k^i-\bar{\omega}_k)(\omega_k^i-\bar{\omega}_k)^{\top} - I \Big)\alpha_{k\Delta_l}^N \\ \nonumber &+2\textrm{\emph{Sym}}\Big(\alpha_{k\Delta_l}^N\Big(\tfrac{1}{N-1}\sum_{i=1}^N(\omega_k^i-\bar{\omega}_k)(\xi_{k\Delta_l}^i-m_{k\Delta_l}^N)^{\top}\Big)(B_{k\Delta_l}^N)^{\top}\Big),  \nonumber
\end{align}
where $\bar{\omega}_k=\tfrac{1}{N}\sum_{i=1}^N \omega_k^i$ and $B_{k\Delta_l}^N=(I+A\Delta_l-P_{k\Delta_l}^NS\Delta_l)$.
\end{proposition}

\begin{proof}
Using the notation of the statement of the result, for the d-EnKF, the recursion \eqref{eq:enkf_ps} becomes, for $i\in\{1,\dots,N\}$:
\begin{equation}\label{eq:enkf_mod_rep}
\xi_{(k+1)\Delta_l}^i = B_{k\Delta_l}^N \xi_{k\Delta_l}^i + U_{k\Delta_l}^N[Y_{(k+1)\Delta_l}-Y_{k\Delta_l}] + \alpha_{k\Delta_l}^N \omega_k^i.
\end{equation}

To establish \eqref{eq:enkf_mean_rec}, we have by using \eqref{eq:enkf_mod_rep} that
\begin{align*}
m_{(k+1)\Delta_l}^N - m_{k\Delta_l}^N   & =  \frac{1}{N}\sum_{i=1}^N\Big\{\xi_{(k+1)\Delta_l}^i-\xi_{k\Delta_l}^i\Big\} \\
& =  \frac{1}{N}\sum_{i=1}^N\Big\{(B_{k\Delta_l}^N -I)\xi_{k\Delta_l}^i + U_{k\Delta_l}^N [Y_{(k+1)\Delta_l}-Y_{k\Delta_l}] + \alpha_{k\Delta_l}^N \omega_k^i\Big\}\\
& =  (B_{k\Delta_l}^N -I)\frac{1}{N}\sum_{i=1}^N\xi_{k\Delta_l}^i + U_{k\Delta_l}^N [Y_{(k+1)\Delta_l}-Y_{k\Delta_l}]  + \alpha_{k\Delta_l}^N\frac{1}{N}\sum_{i=1}^N\omega_k^i\\
& =  \Big(A - P_{k\Delta_l}^NS\Big)\Delta_l m_{k\Delta_l}^N + U_{k\Delta_l}^N [Y_{(k+1)\Delta_l}-Y_{k\Delta_l}] + \alpha_{k\Delta_l}^N\frac{Z_k}{\sqrt{N}}.
\end{align*}

For \eqref{eq:enkf_cov_rec} we have
$$
P_{(k+1)\Delta_l}^N - P_{k\Delta_l}^N = \frac{1}{N-1}\sum_{i=1}^N \xi_{(k+1)\Delta_l}^i(\xi_{(k+1)\Delta_l}^i)^{\top} - \frac{N}{N-1} m_{(k+1)\Delta_l}^N(m_{(k+1)\Delta_l}^N)^{\top}
- P_{k\Delta_l}^N.
$$
Then using \eqref{eq:enkf_mod_rep} to substitute for $\xi_{(k+1)\Delta_l}^i$ and using \eqref{eq:enkf_mean_rec} to substitute for $m_{(k+1)\Delta_l}^N$ we have the decomposition
$$
P_{(k+1)\Delta_l}^N - P_{k\Delta_l}^N = \sum_{j=1}^4 T_j,
$$
where
\begin{align*}
T_1 & =  \frac{1}{N-1}\sum_{i=1}^N \Big\{\Big(B_{k\Delta_l}^N \xi_{k\Delta_l}^i + U_{k\Delta_l}^N[Y_{(k+1)\Delta_l}-Y_{k\Delta_l}]\Big) 
\Big(B_{k\Delta_l}^N \xi_{k\Delta_l}^i + U_{k\Delta_l}^N[Y_{(k+1)\Delta_l}-Y_{k\Delta_l}]\Big)^{\top} 
\Big\}  \\
& - \frac{N}{N-1}\Big(B_{k\Delta_l}^N m_{k\Delta_l}^N + U_{k\Delta_l}^N[Y_{(k+1)\Delta_l}-Y_{k\Delta_l}]\Big) \Big(B_{k\Delta_l}^N m_{k\Delta_l}^N + U_{k\Delta_l}^N[Y_{(k+1)\Delta_l}-Y_{k\Delta_l}]\Big)^{\top} - P_{k\Delta_l}^N,\\
T_2 & =  \frac{1}{N-1}\sum_{i=1}^N \Big\{\Big(B_{k\Delta_l}^N \xi_{k\Delta_l}^i + U_{k\Delta_l}^N[Y_{(k+1)\Delta_l}-Y_{k\Delta_l}]\Big)(\alpha_{k\Delta_l}^N\omega_k^i)^{\top}\Big\} 
\\&-  \frac{N}{N-1}\Big(B_{k\Delta_l}^N m_{k\Delta_l}^N + U_{k\Delta_l}^N[Y_{(k+1)\Delta_l}-Y_{k\Delta_l}]\Big)\\ &\times(\frac{1}{\sqrt{N}}\alpha_{k\Delta_l}^NZ_k)^{\top},
\end{align*}
\begin{align*}
T_3 & =  \frac{1}{N-1}\sum_{i=1}^N\Big\{\alpha_{k\Delta_l}^N\omega_k^i(B_{k\Delta_l}^N \xi_{k\Delta_l}^i+U_{k\Delta_l}^N[Y_{(k+1)\Delta_l}-Y_{k\Delta_l}])^{\top}\Big\}\\ &- 
\frac{N}{N-1}\alpha_{k\Delta_l}^N\frac{Z_k}{\sqrt{N}}(B_{k\Delta_l}^N m_{k\Delta_l}^N+U_{k\Delta_l}^N[Y_{(k+1)\Delta_l}-Y_{k\Delta_l}])^{\top},\\
T_4 & =  \frac{1}{N-1}\sum_{i=1}^N\alpha_{k\Delta_l}^N\omega_k^i(\omega_k^i)^{\top}\alpha_{k\Delta_l}^N -\frac{1}{(N-1)}\alpha_{k\Delta_l}^NZ_k Z_k^{\top}\alpha_{k\Delta_l}^N.
\end{align*}
To complete the proof, we shall simply perform some linear algebra operations on each of the terms $T_1,\dots, T_4$ and sum the resulting expressions.

For $T_1$ we have that, as the cross-product terms will cancel
{
\begin{align}
T_1 & =   B_{k\Delta_l}^N\Big(\frac{1}{N-1}\sum_{i=1}^N \xi_{k\Delta_l}^i(\xi_{k\Delta_l}^i)^{\top}\Big)(B_{k\Delta_l}^N)^{\top} - \frac{N}{N-1}m_{k\Delta_l}^N(m_{k\Delta_l}^N)^{\top} - P^N_{k\Delta_l} \nonumber\\
& =   B_{k\Delta_l}^NP_{k\Delta_l}^N(B_{k\Delta_l}^N)^{\top} - P^N_{k\Delta_l}\nonumber\\
& =  \Big((P_{k\Delta_l}^N-AP_{k\Delta_l}^N\Delta_l - P_{k\Delta_l}^NSP_{k\Delta_l}^N\Delta_l)(I+A^{\top}\Delta_l - SP_{k\Delta_l}^N\Delta_l)\Big) - P^N_{k\Delta_l}\nonumber\\
& =  P_{k\Delta_l}^NA^{\top}\Delta_l - 2P_{k\Delta_l}^NSP_{k\Delta_l}^N\Delta_l + A P_{k\Delta_l}^N\Delta_l + \Delta_l^2(A-P^N_{k\Delta_l}S)P^N_{k\Delta_l}(A^{\top}-SP^N_{k\Delta_l})\nonumber\\
& =   \textrm{Ricc}(P_{k\Delta_l}^N)\Delta_l  + \Delta_l^2(A-P^N_{k\Delta_l}S)P^N_{k\Delta_l}(A^{\top}-SP^N_{k\Delta_l})  - R_1\Delta_l - P_{k\Delta_l}^NSP_{k\Delta_l}^N\Delta_l\nonumber\\
& =  \textrm{Ricc}(P_{k\Delta_l}^N)\Delta_l  + \Delta_l^2(A-P^N_{k\Delta_l}S)P^N_{k\Delta_l}(A^{\top}-SP^N_{k\Delta_l}) - (\alpha_{k\Delta_l}^N)^2.\label{eq:term1}
\end{align}
}
For $T_2$ we have that, as the additional terms cancel
\begin{align}
T_2 & =   B_{k\Delta_l}^N\frac{1}{N-1}\sum_{i=1}^N\xi_{k\Delta_l}^i(\alpha_{k\Delta_l}^N\omega_k^i)^{\top} - \frac{1}{N-1}B_{k\Delta_l}^N m_{k\Delta_l}^N\Big(
\alpha_{k\Delta_l}^N\sum_{i=1}^N\omega_k^i
\Big)^{\top}
\nonumber\\
& =   B_{k\Delta_l}^N\Big(
\frac{{N}}{N-1}\sum_{i=1}^N\xi_{k\Delta_l}^i(\omega_k^i)^{\top} - m_{k\Delta_l}^N(\bar{\omega}_k)^{\top}
\Big)\alpha_{k\Delta_l}^N\nonumber \\
& =   B_{k\Delta_l}^N\Big(\frac{1}{N-1}\sum_{i=1}^N(\xi_{k\Delta_l}^i-m_{k\Delta_l}^N)(\omega_k^i-\bar{\omega}_k)^{\top}\Big)\alpha_{k\Delta_l}^N.\label{eq:term2}
\end{align}
For $T_3$ using almost the same argument
\begin{equation}
T_3 =  \alpha_{k\Delta_l}^N\Big(\frac{1}{N-1}\sum_{i=1}^N(\omega_k^i-\bar{\omega}_k)(\xi_{k\Delta_l}^i-m_{k\Delta_l}^N)^{\top}\Big)(B_{k\Delta_l}^N)^{\top}.\label{eq:term3}
\end{equation}

Finally for $T_4$
\begin{align}
T_4 & =  \alpha_{k\Delta_l}^N\Big(\frac{1}{N-1}\sum_{i=1}^N\omega_k^i(\omega_k^i)^{\top}-\frac{N}{N-1}\bar{\omega}_k(\bar{\omega}_k)^{\top}\Big)\alpha_{k\Delta_l}^N\nonumber\\
& =  \alpha_{k\Delta_l}^N
\Big(\frac{1}{N-1}\sum_{i=1}^N(\omega_k^i-\bar{\omega}_k)(\omega_k^i-\bar{\omega}_k)^{\top}\Big)\alpha_{k\Delta_l}^N. \label{eq:term4}
\end{align}
Summing the R.H.S.~of each of \eqref{eq:term1}-\eqref{eq:term4} yields \eqref{eq:enkf_cov_rec} and the proof is hence completed.
\end{proof}


For a vector $X$ we write $X(i)$ as it $i^{th}-$element and for a matrix $A$ we write $A(i,j)$ as its $(i,j)^{th}$ element. From herein we give  our proofs of $\mathbb{L}_q-$bounds for $q\in[1,\infty)$ as the proof for $q\in(0,1)$ follows by Jensen's inequality; {since this generalization does not complicate matters, we add this condition to our statements}. Throughout our proofs $\mathsf{C}$ will denote a finite constant whose value may change from line-to-line.
The dependencies of the constant on the various quantities associated to the algorithm will be made clear in a given result statement.

{
We will introduce the following assumptions which will hold from herein, but not be added to any statements. 
For a square matrix, $B$ say, we denote by $\mu(B)$ as the maximum eigenvalue of $\textrm{Sym}(B)$.
\begin{enumerate}
\item{We have that $\mu(A)<0$.}
\item{There exists a $\mathsf{C}<+\infty$ such that for any $(k,l)\in\mathbb{N}_0^2$ we have that 
\begin{equation}\label{eq:bound_deter_cov}
\max_{(j_1,j_2)\in\{1,\dots,d_x\}^2}|P_{k\Delta_l}(j_1,j_2)|\leq\mathsf{C}.
\end{equation}
}
\end{enumerate}
We note that 1.~is typically used in the time stability of the hidden diffusion process $X_t$, see for instance \cite{DT18}. In the case of 2.~we expect that it can be verified under 1.,
that $S=\mathsf{C}I$ with $\mathsf{C}$ a positive constant and some controllability and observability assumptions (e.g.~\cite[eq.~(20]{DT18}). Under such assumptions, the Riccati equation has a solution and moreover, by \cite[Proposition 5.3]{DT18} $\mathcal{P}_t$ is exponentially stable w.r.t.~the Frobenius norm; so that this type of bound exists in continuous time. The challenge is then to verify this type of stability on time discretization; we leave this important problem to future work.
}

{The following results are now proved sequentially, to obtain first a type of $\mathbb{L}_q-$bound on the elements of $P_{k\Delta_l}^N-P_{k\Delta_l}$, which converges to zero as $N$ grows (Lemma \ref{lem:p_lq}). This convergence
then permits one to prove the final propagation of chaos result (Theorem \ref{theo:prop}) which is crucial to our analysis in Sections \ref{sec:conv_enkbf} and \ref{sec:theory}. To consider constructing an $\mathbb{L}_q-$bound on the elements $P_{k\Delta_l}^N-P_{k\Delta_l}$ we consider \eqref{eq:enkf_cov_rec}.
This implies that one must show that the appropriate moments exist (Lemmata \ref{lem:y_inc}, \ref{lem:xi_lq}, \ref{lem:al_lq}) and that each element of those equations are sufficiently controlled (Lemmata \ref{lem:p_first}, \ref{lem:p_sec}).}

\begin{lem}\label{lem:y_inc}
For any $q\in(0,\infty)$ there exists a $\mathsf{C}<+\infty$ such that for any $(k,l,j)\in\mathbb{N}_0^2\times\{1,\dots,d_y\}$:{
$$
\mathbb{E}[|[Y_{(k+1)\Delta_l}-Y_{k\Delta_l}](j)|^q]^{1/q} \leq \mathsf{C}\Delta_l^{1/2}.
$$}
\end{lem}

\begin{proof}
Using the integral representation
$$
Y_{(k+1)\Delta_l}-Y_{k\Delta_l} = \int_{k\Delta_l}^{(k+1)\Delta_l}CX_sds + \int_{k\Delta_l}^{(k+1)\Delta_l}R_2^{1/2}dV_s,
$$
one has that by the Minkowski inequality
\begin{align}
\nonumber
\mathbb{E}[|[Y_{(k+1)\Delta_l}-Y_{k\Delta_l}](j)|^q]^{1/q} &\leq  \mathbb{E}\Big[\Big|\sum_{j_1=1}^{d_x} C(j,j_1)\int_{k\Delta_l}^{(k+1)\Delta_l} X_s(j_1)ds\Big|^q\Big]^{1/q} \\ &+  \mathbb{E}\Big[\Big|\sum_{j_1=1}^{d_y}R_2^{1/2}(j,j_1)[V_{(k+1)\Delta_l}-V_{k\Delta_l}](j_1)\Big|^q\Big]^{1/q}  \\
\nonumber
&\leq {\Delta_l^{1-1/q}\sum_{j_1=1}^{d_x} |C(j,j_1)|\mathbb{E}\Big[\int_{k\Delta_l}^{(k+1)\Delta_l}|X_s(j_1)|^qds\Big]^{1/q}} \\ &+
\label{eq:y_inc_1}
+ {\sum_{j_1=1}^{d_y}|R_2^{1/2}(j,j_1)|\mathbb{E}[|[V_{(k+1)\Delta_l}-V_{k\Delta_l}](j_1)|^q]^{1/q},} 
\end{align}
{where we have used Jensen's inequality} to go to \eqref{eq:y_inc_1}. 
By using \cite[eq.~(60)]{DT18} one can deduce that
\begin{equation}\label{eq:y_inc_2}
\mathbb{E}\left[|X_s(j_1)|^q\right] \leq \mathsf{C}.
\end{equation}
Therefore, using \eqref{eq:y_inc_2}  and $\mathbb{E}[|[V_{(k+1)\Delta_l}-V_{k\Delta_l}](j_1)|^q]^{1/q}\leq\mathsf{C}\Delta_l^{1/2}$ in \eqref{eq:y_inc_1} we have that:
$$
\mathbb{E}[|[Y_{(k+1)\Delta_l}-Y_{k\Delta_l}](j)|^q]^{1/q} \leq \mathsf{C}\Delta_l^{1/2}
$$
and the proof is complete.

\end{proof}

{
\begin{lem}\label{lem:xi_lq}
For any $(q,k,l)\in(0,\infty)\times\mathbb{N}_0^2$ there exists a $\mathsf{C}<+\infty$ such that for any $(N,i)\in\mathbb{N}_0^2\times\{1,\dots,d_x\}\times\{2,3,\dots\}\times\{1,\dots,N\}$:
$$
\mathbb{E}[|\xi_{k\Delta_l}^i(j)|^q]^{1/q} \leq \mathsf{C}.
$$
\end{lem}
}

\begin{proof}
The proof consists of deriving a recursion on $k$ noting that the case $k=0$ is trivial as $\xi^i_0(j)\stackrel{\textrm{i.i.d.}}{\sim}\mathcal{N}(\mathcal{M}_0(j),\mathcal{P}_0(j,j))$, $i\in\{1,\dots,N\}$.

We begin with the following upper-bound, that can be derived by combining \eqref{eq:enkf_mod_rep} along with the Minkowski inequality,
$$
\mathbb{E}[|\xi_{k\Delta_l}^i(j)|^q]^{1/q} \leq \sum_{j=1}^3 T_j,
$$
where
\begin{align*}
T_1 & =  \mathbb{E}\Big[\Big|\sum_{j_1=1}^{d_x} \Big\{\mathbb{I}_{\{j\}}(j_1)+A(j,j_1)\Delta_l-\sum_{j_2=1}^{d_x} P_{k\Delta_l}^N(j,j_2)S(j_2,j_1)\Delta_l\Big\}\xi_{k\Delta_l}^{i}(j_1)\Big|^q\Big]^{1/q},\\
T_2 & =  \mathbb{E}\Big[\Big|\sum_{j_1=1}^{d_x}\sum_{j_2=1}^{d_x}P_{k\Delta_l}^N(j,j_1)\tilde{C}(j_1,j_2)[Y_{(k+1)\Delta_l}-Y_{k\Delta_l}](j_2)\Big|^q\Big]^{1/q},\\
T_3 & =  \mathbb{E}\Big[\Big|\sum_{j_1=1}^{d_x}\alpha_{k\Delta_l}^N(j,j_1) \omega_k^i(j_1)\Big|^q\Big]^{1/q},
\end{align*}
{where $\mathbb{I}_{\{j\}}(j_1)$ is the indicator of the set $\{j\}$, i.e.~$\mathbb{I}_{\{j\}}(j)=1$}, and $\tilde{C}=C^{\top} R_2^{-1}$. We will bound each of the three terms $T_j$ $j\in\{1,2,3\}$ and sum the bounds to obtain our recursion.

For the term $T_1$ we have by several applications of the Minkowski inequality
\begin{equation}\label{eq:xi_lq_1}
T_1 \leq \sum_{j_1=1}^{d_x}\Big\{|\mathbb{I}_{\{j\}}(j_1)+A(j,j_1)\Delta_l|\mathbb{E}[|\xi_{k\Delta_l}^{i}(j_1)|^q]^{1/q}+\sum_{j_2=1}^{d_x}|S(j_2,j_1)|\Delta_l
\mathbb{E}[|P_{k\Delta_l}^N(j,j_2)\xi_{k\Delta_l}^{i}(j_1)|^q]^{1/q}
\Big\}.
\end{equation}
We will now focus on the right most expectation of the R.H.S.~of \eqref{eq:xi_lq_1}. Applying Young's inequality and, for $\mathbb{E}[|P_{k\Delta_l}^N(j,j_2)|^{2q}]^{1/(2q)}$ using Minkowski and the exchangeability of the particle system, we have the upper-bound
$$
\mathbb{E}[|P_{k\Delta_l}^N(j,j_2)\xi_{k\Delta_l}^{i}(j_1)|^q]^{1/q} \leq  \mathsf{C}\Big(\mathbb{E}[|[\xi_{k\Delta_l}^{i}-m_{k\Delta_l}](j_2)
[\xi_{k\Delta_l}^{i}-m_{k\Delta_l}](j)|^{2q}]^{1/(2q)}+\mathbb{E}[|\xi_{k\Delta_l}^{i}(j_1)|^{2q}]^{1/(2q)}\Big).
$$
Applying Young's inequality to the left-most quantity on the R.H.S., one can deduce, using the exchangeability of the particle system and Jensen for
the right-most quantity on the R.H.S.~that
$$
\mathbb{E}[|P_{k\Delta_l}^N(j,j_2)\xi_{k\Delta_l}^{i}(j_1)|^q]^{1/q} \leq \mathsf{C}\max_{j_1\in\{1,\dots,d_x\}}\{\mathbb{E}[|\xi_{k\Delta_l}^{i}(j_1)|^{4q}]^{1/(4q)}\}.
$$
Thus returning to \eqref{eq:xi_lq_1} we can deduce the upper-bound
\begin{align}
\nonumber
T_1 & \leq \mathsf{C}\max_{(j_1,j_2)\in\{1,\dots,d_x\}^2}\{|\mathbb{I}_{\{j_2\}}(j_1)+A(j_2,j_1)\Delta_l|+|S(j_2,j_1)|\}\max_{j_1\in\{1,\dots,d_x\}}\{\mathbb{E}[|\xi_{k\Delta_l}^{i}(j_1)|^{4q}]^{1/(4q)}\}
\nonumber\\
& \leq  \mathsf{C}\max_{j_1\in\{1,\dots,d_x\}}\{\mathbb{E}[|\xi_{k\Delta_l}^{i}(j_1)|^{4q}]^{1/(4q)}\}.
\label{eq:xi_lq_2}
\end{align}

For $T_2$ by Minkowski and Cauchy-Schwarz
$$
T_2 \leq \sum_{j_1=1}^{d_x}\sum_{j_2=1}^{d_x}|\tilde{C}(j_1,j_2)|\mathbb{E}[|P_{k\Delta_l}^N(j,j_2)|^{2q}]^{1/(2q)}\mathbb{E}[|[Y_{(k+1)\Delta_l}-Y_{k\Delta_l}](j_2)|^{2q}]^{1/(2q)},
$$
then applying Lemma \ref{lem:y_inc} along with using the above calculations for $T_1$ to control $\mathbb{E}[|P_{k\Delta_l}^N(j,j_2)|^{2q}]^{1/(2q)}$, we have the upper-bound:
{
\begin{equation}
T_2 \leq \mathsf{C}\max_{j_1\in\{1,\dots,d_x\}}\{\mathbb{E}[|\xi_{k\Delta_l}^{i}(j_1)|^{4q}]^{1/(4q)}\}.
\label{eq:xi_lq_3}
\end{equation}}

For $T_3$ by using the fact that, conditional on the $\sigma-$field generated by the particle system and the observations {up-to} time $k\Delta_l$
$[\alpha_{k\Delta_l}^N\omega_k^i](j)\sim\mathcal{N}(0,[\alpha_{k\Delta_l}^N\alpha_{k\Delta_l}^N](j,j))$, we have that 
$$
T_3 \leq \mathsf{C}\Delta_l^{1/2}\mathbb{E}[|(R_1+P_{k\Delta_l}^NSP_{k\Delta_l}^N)(j,j)|^{q/2}]^{1/q}.
$$
Using the above calculations for $T_1$ to control $\mathbb{E}[|P_{k\Delta_l}^N(j,j_2)|^{2q}]^{1/(2q)}$, one can use very {standard} calculations to deduce that
\begin{equation}
T_3 \leq
 \mathsf{C}\left(\max_{j_1\in\{1,\dots,d_x\}}\{\mathbb{E}[|\xi_{k\Delta_l}^{i}(j_1)|^{4q}]^{1/(4q)}\}+\Delta_l^{1/2}\right).
\label{eq:xi_lq_4}
\end{equation}
Now combining \eqref{eq:xi_lq_2}-\eqref{eq:xi_lq_4} we have the upper-bound
$$
\max_{j_1\in\{1,\dots,d_x\}}\{\mathbb{E}[|\xi_{(k+1)\Delta_l}^{i}(j_1)|^{q}]^{1/(q)}\} \leq
\mathsf{C}\left(\max_{j_1\in\{1,\dots,d_x\}}\{\mathbb{E}[|\xi_{k\Delta_l}^{i}(j_1)|^{4q}]^{1/(4q)}\}+\Delta_l^{1/2}\right).
 $$
Applying this recursively yields that
\begin{eqnarray*}
\max_{j_1\in\{1,\dots,d_x\}}\{\mathbb{E}[|\xi_{(k+1)\Delta_l}^{i}(j_1)|^{q}]^{1/(q)}\} \leq  \\  \mathsf{C}\left(
\max_{j_1\in\{1,\dots,d_x\}}\{\mathbb{E}[|\xi_{0}^{i}(j_1)|^{4q(k+1)}]^{1/(4q(k+1)}\}+(k+1)\Delta_l^{1/2}\right)
 \end{eqnarray*}
 and the proof is hence completed.

\end{proof}


\begin{lem}\label{lem:al_lq}
{For any $(q,k,l)\in(0,\infty)\times\mathbb{N}_0^2$ there exists a $\mathsf{C}<+\infty$ such that for any $(j_1,j_2,N)\in\{1,\dots,d_x\}^2\times\{2,3,\dots\}$:
$$
\mathbb{E}[|\alpha_{k\Delta_l}^N(j_1,j_2)|^q]^{1/q} \leq \mathsf{C}\Delta_l^{1/2}.
$$}
\end{lem}

\begin{proof}
By the representation of $\alpha_{k\Delta_l}^N$, {we have
$$
\mathbb{E}[|(\alpha_{k\Delta_l}^N)^2(j_1,j_2)|^q]^{1/q} = \Delta_l\mathbb{E}[|(R_1 + P_{k\Delta_l}^NSP_{k\Delta_l}^N)(j_1,j_2)|^q]^{1/q}.
$$
}
So using Lemma \ref{lem:xi_lq}, {it follows}, that 
for any $(q,k,l)\in(0,\infty)\times\mathbb{N}_0^2$ such that for any $(l,j_1,j_2,N)\in\mathbb{N}_0\times\{1,\dots,d_x\}^2\times\{2,3,\dots\}$:
\begin{equation}\label{eq:al_lq_1}
\mathbb{E}[|(\alpha_{k\Delta_l}^N)^2(j_1,j_2)|^q]^{1/q} \leq \mathsf{C}\Delta_l.
\end{equation}
The eigen-decomposition of $\Delta_l^{-1}(\alpha_{k\Delta_l}^N)^2$, is written as 
$$X((\alpha_{k\Delta_l}^N)^2)\Lambda((\alpha_{k\Delta_l}^N)^2)X((\alpha_{k\Delta_l}^N)^2)^{\top}=\Delta_l^{-1}(\alpha_{k\Delta_l}^N)^2,$$ 
and thus 
$$
X((\alpha_{k\Delta_l}^N)^2)\Lambda((\alpha_{k\Delta_l}^N)^2)^{1/2}X((\alpha_{k\Delta_l}^N)^2)^{\top}=\Delta_l^{-1/2}(\alpha_{k\Delta_l}^N).
$$
Now as $\Delta_l^{-1}(\alpha_{k\Delta_l}^N)^2$ is the sum of a positive definite symmetric matrix, and a positive semi-definite symmetric matrix, we know that minimum eigenvalue
of $\Delta_l^{-1}(\alpha_{k\Delta_l}^N)^2$ is lower-bounded by a deterministic and positive constant, (the minimum eigenvalue of $R_1$), i.e.  
\begin{equation}\label{eq:al_lq_2}
\Lambda((\alpha_{k\Delta_l}^N)^2)(1,1) \geq \mathsf{\bar{C}}.
\end{equation}
Thus, for the diagonal elements of $\Delta_l^{-1/2}\alpha_{k\Delta_l}^N$, $j_1\in\{1,\dots,d_x\}$:
\begin{align*}
\Delta_l^{-1/2}\alpha_{k\Delta_l}^N(j_1,j_1) & =  \sum_{j=1}^{d_x} \Lambda((\alpha_{k\Delta_l}^N)^2)(j,j)^{1/2}X((\alpha_{k\Delta_l}^N)^2)(j_1,j)^2\\
& \leq  \mathsf{C}\sum_{j=1}^{d_x} \Lambda((\alpha_{k\Delta_l}^N)^2)(j,j)X((\alpha_{k\Delta_l}^N)^2)(j_1,j)^2,
\end{align*}
where $\mathsf{C}=1/\sqrt{\mathsf{\bar{C}}}$ and $\mathsf{\bar{C}}$ is as \eqref{eq:al_lq_2}. Hence it follows that by \eqref{eq:al_lq_1}
\begin{equation}\label{eq:al_lq_3}
\mathbb{E}[|\alpha_{k\Delta_l}^N(j_1,j_1)|^q]^{1/q} \leq \mathsf{C}\Delta_l^{1/2}.
\end{equation}
For the off-diagonal elements of $\alpha_{k\Delta_l}^N$ one can simply use \eqref{eq:al_lq_1} and \eqref{eq:al_lq_3} combined with the equations for $(j_1,j_2)\in\{1,\dots,d_x\}^2$
$$
(\alpha_{k\Delta_l}^N)^2(j_1,j_2) = \sum_{j=1}^{d_x}\alpha_{k\Delta_l}^N(j_1,j)\alpha_{k\Delta_l}^N(j,j_2).
$$
As we know all the moments of $(\alpha_{k\Delta_l}^N)^2(j_1,j_2)$ and $\alpha_{k\Delta_l}^N(j_1,j_2)$ are bounded and we have ${d_x(d_x+1)/2}$ equations with ${d_x(d_x-1)/2}$
quantities to bound, the proof is then completed by {standard} algebraic and probabilistic calculations.
\end{proof}

\begin{lem}\label{lem:p_first}
{
For any $(q,k,l)\in(0,\infty)\times\mathbb{N}_0$ there exists a $\mathsf{C}<+\infty$ such that for any $(j_1,j_2,N)\in\{1,\dots,d_x\}^2\times\{2,3,\dots\}$:
$$
\mathbb{E}\Big[\Big|\Big(\alpha_{k\Delta_l}^N
\Big(\tfrac{1}{N-1}\sum_{i=1}^N(\omega_k^i-\bar{\omega}_k)(\omega_k^i-\bar{\omega}_k)^{\top} - I \Big)\alpha_{k\Delta_l}^N\Big)(j_1,j_2)\Big|^q\Big]^{1/q} \leq \frac{\mathsf{C}\Delta_l}{\sqrt{N}}.
$$}
\end{lem}

\begin{proof}
We have
$$
\Big(\alpha_{k\Delta_l}^N(\tfrac{1}{N-1}\sum_{i=1}^N(\omega_k^i-\bar{\omega}_k)(\omega_k^i-\bar{\omega}_k)^{\top} - I \Big)\alpha_{k\Delta_l}^N\Big)(j_1,j_2)
= 
$$
\begin{align*}
\sum_{j_4=1}^{d_x}\sum_{j_3=1}^{d_x}\alpha_{k\Delta_l}^N(j_1,j_3)\alpha_{k\Delta_l}^N(j_4,j_2)\tfrac{1}{N-1}\sum_{i=1}^N(\omega_k^i-\bar{\omega}_k)(j_3)(\omega_k^i -\bar{\omega}_k)(j_4) 
 - \sum_{j_3=1}^{d_x}\alpha_{k\Delta_l}^N(j_1,j_3)\alpha_{k\Delta_l}^N(j_3,j_2) = T_1 + T_2,
\end{align*}
where
\begin{align*}
T_1 & :=  \sum_{j_3=1}^{d_x}\alpha_{k\Delta_l}^N(j_1,j_3)\alpha_{k\Delta_l}^N(j_3,j_2)\Big(\tfrac{1}{N-1}\sum_{i=1}^N(\omega_k^i-\bar{\omega}_k)(j_3)^2-1\Big), \\
T_2 & :=   
\sum_{j_4=1}^{d_x}\sum_{j_3=1}^{d_x}\mathbb{I}_{D^c}(j_3,j_4)
\alpha_{k\Delta_l}^N(j_1,j_3)\alpha_{k\Delta_l}^N(j_4,j_2)\tfrac{1}{N-1}\sum_{i=1}^N(\omega_k^i-\bar{\omega}_k)(j_3)(\omega_k^i-\bar{\omega}_k)(j_4),
\end{align*}
and $D=\{(j_3,j_4)\in\{1,\dots,d_x\}^2:j_3=j_4\}$. So by Minkowski, we can bound the $\mathbb{L}_q-$moments of $T_1$ and $T_2$ independently and sum the bounds.

For $T_1$ by Minkowksi and the independence of $\omega_k^i$
$$
\mathbb{E}[|T_1|^q]^{1/q} \leq  \sum_{j_3=1}^{d_x}\Big(\mathbb{E}[|\alpha_{k\Delta_l}^N(j_1,j_3)\alpha_{k\Delta_l}^N(j_3,j_2)|^q]\mathbb{E}\Big[\Big|
\tfrac{1}{N-1}\sum_{i=1}^N(\omega_k^i-\bar{\omega}_k)(j_3)^2-1\Big|^q\Big]\Big)^{1/q},
$$
by using {standard} calculations along with the Marcinkiewicz-Zygmund inequality and the moments of Gaussian random variables
$$
\mathbb{E}[|T_1|^q]^{1/q} \leq  \frac{\mathsf{C}}{\sqrt{N}}\sum_{j_3=1}^{d_x}\mathbb{E}[|\alpha_{k\Delta_l}^N(j_1,j_3)\alpha_{k\Delta_l}^N(j_3,j_2)|^q]^{1/q}.
$$
Application of Cauchy-Schwarz and Lemma \ref{lem:al_lq} yields
\begin{equation}\label{eq:quad_lq_1}
\mathbb{E}[|T_1|^q]^{1/q} \leq  \frac{\mathsf{C}\Delta_l}{\sqrt{N}}.
\end{equation}

For $T_2$ by Minkowksi and the independence of $\omega_k^i$
\begin{align*}
\mathbb{E}[|T_2|^q]^{1/q} &\leq  \sum_{j_4=1}^{d_x}\sum_{j_3=1}^{d_x}\mathbb{I}_{D^c}(j_3,j_4)\Big(\mathbb{E}[|\alpha_{k\Delta_l}^N(j_1,j_3)\alpha_{k\Delta_l}^N(j_4,j_2)|^q]
 \mathbb{E}\Big[\Big| \tfrac{1}{N-1}\sum_{i=1}^N(\omega_k^i-\bar{\omega}_k)(j_3)(\omega_k^i-\bar{\omega}_k)(j_4)\Big|^q\Big]\Big)^{1/q}.
\end{align*}
One can then use a similar argument to $T_1$ to deduce that
\begin{equation}\label{eq:quad_lq_2}
\mathbb{E}[|T_2|^q]^{1/q} \leq  \frac{\mathsf{C}\Delta_l}{\sqrt{N}}.
\end{equation}
Summing \eqref{eq:quad_lq_1} and \eqref{eq:quad_lq_2} conclude the proof.

\end{proof}


\begin{lem}\label{lem:p_sec}
{For any $(q,k,l)\in(0,\infty)\times\mathbb{N}_0^2$ there exists a $\mathsf{C}<+\infty$ such that for any $(j_1,j_2,N)\in\{1,\dots,d_x\}^2\times\{2,3,\dots\}$:
$$
\mathbb{E}\Big[\Big|\Big(\alpha_{k\Delta_l}^N
\Big(\tfrac{1}{N-1}\sum_{i=1}^N(\omega_k^i-\bar{\omega}_k)(\xi_{k\Delta_l}^i-m_{k\Delta_l}^N)^{\top}\Big)(B_{k\Delta_l}^N)^{\top}\Big)(j_1,j_2)\Big|^q\Big]^{1/q} \leq \frac{\mathsf{C}\Delta_l^{1/2}}{\sqrt{N}}.
$$}
\end{lem}

\begin{proof}
We first remark that via Lemma \ref{lem:xi_lq} it easily follows that 
for any $(q,k,l)\in(0,\infty)\times\mathbb{N}_0^2$ there exists a $\mathsf{C}\in[1,\infty)$ such that for any $(j_1,j_2,N)\in\{1,\dots,d_x\}^2\times\{2,3,\dots\}$:
\begin{equation}\label{eq:quad1_lq_1}
\mathbb{E}[|B_{k\Delta_l}^N(j_1,j_2)|^q]^{1/q} \leq \mathsf{C}.
\end{equation}

Now, we note that 
$$
\alpha_{k\Delta_l}^N
\Big(\tfrac{1}{N-1}\sum_{i=1}^N(\omega_k^i-\bar{\omega}_k)(\xi_{k\Delta_l}^i-m_{k\Delta_l}^N)^{\top}\Big)(B_{k\Delta_l}^N)^{\top}(j_1,j_2)
= 
$$
$$
\sum_{j_4=1}^{d_x}\sum_{j_3=1}^{d_x}\alpha_{k\Delta_l}^N(j_1,j_3)B_{k\Delta_l}^N(j_2,j_4)\tfrac{1}{N-1}\sum_{i=1}^N(\omega_k^i-\bar{\omega}_k)(j_3)(\xi_{k\Delta_l}^i-m_{k\Delta_l}^N)(j_4). 
$$
Thus we have by Minkowski
$$
\mathbb{E}\Big[\Big|\Big(\alpha_{k\Delta_l}^N
\Big(\tfrac{1}{N-1}\sum_{i=1}^N(\omega_k^i-\bar{\omega}_k)(\xi_{k\Delta_l}^i-m_{k\Delta_l}^N)^{\top}\Big)(B_{k\Delta_l}^N)^{\top}\Big)(j_1,j_2)\Big|^q\Big]^{1/q} \leq
$$
$$
\sum_{j_4=1}^{d_x}\sum_{j_3=1}^{d_x}\mathbb{E}\Big[\Big|
\alpha_{k\Delta_l}^N(j_1,j_3)B_{k\Delta_l}^N(j_2,j_4)\tfrac{1}{N-1}\sum_{i=1}^N(\omega_k^i-\bar{\omega}_k)(j_3)(\xi_{k\Delta_l}^i-m_{k\Delta_l}^N)(j_4)\Big|^q\Big]^{1/q}.
$$
Now if $\mathscr{F}^N_{k\Delta_l}$ is the $\sigma-$field generated by the particle system and observations {up-to} time $k\Delta_l$
$$
\mathbb{E}\Big[\Big|\tfrac{1}{N-1}\sum_{i=1}^N(\omega_k^i-\bar{\omega}_k)(j_3)(\xi_{k\Delta_l}^i-m_{k\Delta_l}^N)(j_4)\Big|^q\Big|\mathscr{F}^N_{k\Delta_l}\Big]
=
$$
$$
\mathbb{E}\Big[\Big|\tfrac{1}{N}\sum_{i=1}^N\omega_k^i(j_3)\frac{N(\xi_{k\Delta_l}^i-m_{k\Delta_l}^N)(j_4)}{N-1}
- \Big(\tfrac{1}{N}\sum_{i=1}^N\frac{N(\xi_{k\Delta_l}^i-m_{k\Delta_l}^N)(j_4)}{N-1}\Big)\tfrac{1}{N}\sum_{i=1}^N\omega_k^i(j_3)
\Big|^q\Big|\mathscr{F}^N_{k\Delta_l}\Big] \leq
$$
\begin{align*}
\mathsf{C}\Big(\mathbb{E}\Big[\Big|\tfrac{1}{N}\sum_{i=1}^N\omega_k^i(j_3)\frac{N(\xi_{k\Delta_l}^i-m_{k\Delta_l}^N)(j_4)}{N-1}\Big|^q\Big|\mathscr{F}^N_{k\Delta_l}\Big]
+ 
\mathbb{E}\Big[\Big|\Big(\tfrac{1}{N}\sum_{i=1}^N\frac{N(\xi_{k\Delta_l}^i-m_{k\Delta_l}^N)(j_4)}{N-1}\Big)\tfrac{1}{N}\sum_{i=1}^N\omega_k^i(j_3)
\Big|^q\Big|\mathscr{F}^N_{k\Delta_l}\Big]
\Big).
\end{align*}
Then applying the conditional Marcinkiewicz-Zygmund inequality, almost surely
$$
\mathbb{E}\Big[\Big|\tfrac{1}{N-1}\sum_{i=1}^N(\omega_k^i-\bar{\omega}_k)(j_3)(\xi_{k\Delta_l}^i-m_{k\Delta_l}^N)(j_4)\Big|^q\Big|\mathscr{F}^N_{k\Delta_l}\Big]\leq
$$
$$
\frac{\mathsf{C}}{N^{q/2+1}}\sum_{i=1}^N |(\xi_{k\Delta_l}^i-m_{k\Delta_l}^N)(j_4)|^q + \frac{\mathsf{C}}{N^{q/2}}\Big|\Big(\tfrac{1}{N}\sum_{i=1}^N\frac{N(\xi_{k\Delta_l}^i-m_{k\Delta_l}^N)(j_4)}{N-1}\Big)\Big|^q,
$$
where $\mathsf{C}$ is a deterministic constant, so that we have shown that 
$$
\mathbb{E}\Big[\Big|\Big(\alpha_{k\Delta_l}^N
\Big(\tfrac{1}{N-1}\sum_{i=1}^N(\omega_k^i-\bar{\omega}_k)(\xi_{k\Delta_l}^i-m_{k\Delta_l}^N)^{\top}\Big)(B_{k\Delta_l}^N)^{\top}\Big)(j_1,j_2)\Big|^q\Big]^{1/q} \leq
$$
\begin{align*}
\sum_{j_4=1}^{d_x}\sum_{j_3=1}^{d_x}\Big(\mathbb{E}\Big[\Big||\alpha_{k\Delta_l}^N(j_1,j_3)B_{k\Delta_l}^N(j_2,j_4)|^q\Big\{
\frac{\mathsf{C}}{N^{q/2+1}}\sum_{i=1}^N |(\xi_{k\Delta_l}^i-m_{k\Delta_l}^N)(j_4)|^q  \\+  \frac{\mathsf{C}}{N^{q/2}}\Big|\Big(\tfrac{1}{N}\sum_{i=1}^N\frac{N(\xi_{k\Delta_l}^i-m_{k\Delta_l}^N)(j_4)}{N-1}\Big)\Big|^q\Big\}\Big]\Big)^{1/q}.
\end{align*}
The result can now be concluded by standard calculations combined with Lemmata \ref{lem:xi_lq}-\ref{lem:al_lq} and \eqref{eq:quad1_lq_1}.

\end{proof}
%

\begin{lem}\label{lem:p_lq}
{For any $(q,k,l)\in(0,\infty)\times\mathbb{N}_0^2$ there exists a $\mathsf{C}<+\infty$ such that for any $(j_1,j_2,N)\in\{1,\dots,d_x\}^2\times\{2,3,\dots\}$:
$$
\mathbb{E}\Big[\Big|P_{k\Delta_l}^N(j_1,j_2)-P_{k\Delta_l}(j_1,j_2)\Big|^q\Big]^{1/q} \leq \frac{\mathsf{C}}{\sqrt{N}}.
$$}
\end{lem}

\begin{proof}
The proof is by constructing a recursive in time bound and noting that the case $k=0$ follows by standard results for i.i.d.~sampling and Gaussian random variables. 

We start by noting that by \eqref{eq:iid_cov_rec} and \eqref{eq:enkf_cov_rec} of Proposition \ref{prop:denkf_rec} 
$$
P_{(k+1)\Delta_l}^N(j_1,j_2)-P_{(k+1)\Delta_l}(j_1,j_2) = \sum_{j=1}^6 T_j,
$$
where
\begin{align*}
T_1 & =  P_{k\Delta_l}^N(j_1,j_2)-P_{k\Delta_l}(j_1,j_2),\\
T_2 & =  \Big(\textrm{Ricc}(P_{k\Delta_l}^N)(j_1,j_2)-\textrm{Ricc}(P_{k\Delta_l})(j_1,j_2)\Big)\Delta_l,\\
T_3 & = \Big(\textrm{SRicc}(P_{k\Delta_l}^N)(j_1,j_2)-\textrm{SRicc}(P_{k\Delta_l})(j_1,j_2)\Big)\Delta_l^2,\\
T_4 & =  \Big(\alpha_{k\Delta_l}^N
\Big(\tfrac{1}{N-1}\sum_{i=1}^N(\omega_k^i-\bar{\omega}_k)(\omega_k^i-\bar{\omega}_k)^{\top} - I \Big)\alpha_{k\Delta_l}^N\Big)(j_1,j_2),\\
T_5 & =  \Big(\alpha_{k\Delta_l}^N
\Big(\tfrac{1}{N-1}\sum_{i=1}^N(\omega_k^i-\bar{\omega}_k)(\xi_{k\Delta_l}^i-m_{k\Delta_l}^N)^{\top}\Big)(B_{k\Delta_l}^N)^{\top}\Big)(j_1,j_2), \\
T_6 & =  \Big(\alpha_{k\Delta_l}^N
\Big(\tfrac{1}{N-1}\sum_{i=1}^N(\omega_k^i-\bar{\omega}_k)(\xi_{k\Delta_l}^i-m_{k\Delta_l}^N)^{\top}\Big)(B_{k\Delta_l}^N)^{\top}\Big)(j_2,j_1).
\end{align*}
To construct our recursion, via Minkowski, it is sufficient to obtain $\mathbb{L}_q-$bounds for each of the terms $T_1,\dots,T_6$. We note
\begin{equation}\label{eq:p_lq_1}
\mathbb{E}[|T_1|^q]^{1/q} = \mathbb{E}\Big[\Big|P_{k\Delta_l}^N(j_1,j_2)-P_{k\Delta_l}(j_1,j_2)\Big|^q\Big]^{1/q},
\end{equation}
and for $j\in\{4,5,6\}$ by Lemmata \ref{lem:p_first} and \ref{lem:p_sec}:
\begin{equation}\label{eq:p_lq_2}
{\mathbb{E}[|T_j|^q]^{1/q} \leq \frac{\mathsf{C}\Delta_l^{1/2}}{\sqrt{N}}.}
\end{equation}
As a result, we will just consider $T_2$ and $T_3$ only.

For $T_2$ we have that
\begin{align*}
T_2 & =   \sum_{j_3=1}^{d_x}\Big(A(j_1,j_3)[P_{k\Delta_l}^N(j_3,j_2)-P_{k\Delta_l}(j_3,j_2)]
+A(j_2,j_3)[P_{k\Delta_l}^N(j_1,j_3)-P_{k\Delta_l}(j_1,j_3)]\Big) + 
\\ &  \sum_{j_4=1}^{d_x}  \sum_{j_3=1}^{d_x}S(j_3,j_4)\Big(P_{k\Delta_l}^N(j_1,j_3)P_{k\Delta_l}^N(j_4,j_2)-P_{k\Delta_l}(j_1,j_3)P_{k\Delta_l}(j_4,j_2)\Big).
\end{align*}
{By applying} \eqref{eq:bound_deter_cov}, {standard} calculations yield that (note that $\mathbb{E}[P_{k\Delta_l}^N(j_1,j_2)|^q]\leq \mathsf{C}$ by Lemma \ref{lem:xi_lq})
\begin{equation}\label{eq:p_lq_3}
\mathbb{E}[|T_2|^q]^{1/q} \leq \mathsf{C} \max_{(j_1,j_2)\in\{1,\dots,d_x\}^2}\mathbb{E}\Big[\Big|P_{k\Delta_l}^N(j_1,j_2)-P_{k\Delta_l}(j_1,j_2)\Big|^{2q}\Big]^{1/(2q)}.
\end{equation}
One can use a similar approach to deduce that also
\begin{equation}\label{eq:p_lq_4}
\mathbb{E}[|T_3|^q]^{1/q} \leq \mathsf{C} \max_{(j_1,j_2)\in\{1,\dots,d_x\}^2}\mathbb{E}\Big[\Big|P_{k\Delta_l}^N(j_1,j_2)-P_{k\Delta_l}(j_1,j_2)\Big|^{2q}\Big]^{1/(2q)}.
\end{equation}
Therefore collecting together the bounds \eqref{eq:p_lq_1}-\eqref{eq:p_lq_4} we have that 
$$
\max_{(j_1,j_2)\in\{1,\dots,d_x\}^2}\mathbb{E}\Big[\Big|P_{(k+1)\Delta_l}^N(j_1,j_2)-P_{(k+1)\Delta_l}(j_1,j_2)\Big|^q\Big]^{1/q}  \leq 
$$
$$
\mathsf{C} \Big(\max_{(j_1,j_2)\in\{1,\dots,d_x\}^2}\mathbb{E}\Big[\Big|P_{k\Delta_l}^N(j_1,j_2)-P_{k\Delta_l}(j_1,j_2)\Big|^{2q}\Big]^{1/(2q)}
 + \frac{\Delta_l^{1/2}}{\sqrt{N}}\Big).
$$
Hence we have that
$$
\max_{(j_1,j_2)\in\{1,\dots,d_x\}^2}\mathbb{E}\Big[\Big|P_{(k+1)\Delta_l}^N(j_1,j_2)-P_{(k+1)\Delta_l}(j_1,j_2)\Big|^q\Big]^{1/q} \leq 
$$
$${
\mathsf{C}\Big( \max_{(j_1,j_2)\in\{1,\dots,d_x\}^2}\mathbb{E}\Big[\Big|P_{0}^N(j_1,j_2)-P_{0}(j_1,j_2)\Big|^{2q(k+1)}\Big]^{1/(2q(k+1))}
+ 
\frac{(k+1)\Delta_l^{1/2}}{\sqrt{N}}\Big),}
$$
from which the proof the proof is easily completed.

\end{proof}

\begin{lem}{
For any $(q,k,l)\in(0,\infty)\times\mathbb{N}_0^2$ there exists a $\mathsf{C}<+\infty$ such that for any $(j,N)\in\{1,\dots,d_x\}\times\{2,3,\dots\}$:
$$
\mathbb{E}\Big[\Big|m_{t+k_1\Delta_l}^N(j)-m_{t+k_1\Delta_l}(j)\Big|^q\Big]^{1/q} \leq \frac{\mathsf{C}}{\sqrt{N}}.
$$}
\end{lem}

\begin{proof}
The proof is more-or-less similar to that of Lemma \ref{lem:p_lq}. The difference is of course that one must consider
\eqref{eq:iid_mean_rec} and \eqref{eq:enkf_mean_rec} and that one must apply Lemmata \ref{lem:p_lq} and \ref{lem:y_inc}.
As the argument is fairly simple given this description, it is omitted for brevity.

\end{proof}


We now state our main result of interest, {which is a type of non-uniform} (in time) propagation of chaos result.

\begin{theorem}\label{theo:prop}
{
For any $(q,k,l)\in(0,\infty)\times\mathbb{N}_0^2$ there exists a $\mathsf{C}<+\infty$ such that for any $(j,N,i)\in\{1,\dots,d_x\}\times\{2,3,\dots\}\times\{1,\dots,N\}$:
$$
\mathbb{E}\Big[\Big|\xi_{k\Delta_l}^i(j)-\zeta_{k\Delta_l}^i(j)\Big|^q\Big]^{1/q} \leq \frac{\mathsf{C}}{\sqrt{N}}.
$$}
\end{theorem}

\begin{proof}
We first note that \eqref{eq:enkf_is} can be rewritten as
\begin{equation}\label{eq:enkf_mod_iid}
\zeta_{(k+1)\Delta_l}^i = B_{k\Delta_l} \zeta_{k\Delta_l}^i + U_{k\Delta_l} [Y_{(k+1)\Delta_l}-Y_{k\Delta_l}] + \alpha_{k\Delta_l}\omega_k^i.
\end{equation}
where $\alpha_{k\Delta_l}=(R_1+P_{k\Delta_l}SP_{k\Delta_l})^{1/2}\Delta_l^{1/2}$ and $B_{k\Delta_l}=I+A\Delta_l-P_{k\Delta_l}S\Delta_l$. Thus,
we have that by applying \eqref{eq:enkf_mod_rep} 
\begin{equation}\label{eq:enkf_prop_decomp}
\xi_{(k+1)\Delta_l}^i(j)-\zeta_{(k+1)\Delta_l}^i(j) = \sum_{j_1=1}^3 T_{j_1},
\end{equation}
where
\begin{align*}
T_1 & =  \sum_{j_1=1}^{d_x}\Big(B_{k\Delta_l}^N(j,j_1)\xi_{k\Delta_l}^i(j_1) -B_{k\Delta_l}(j,j_1)\zeta_{k\Delta_l}^i(j_1)\Big),\\
T_2 & =  \sum_{j_2=1}^{d_y}\sum_{j_1=1}^{d_x}\Big([P_{k\Delta_l}^N(j,j_1)-P_{k\Delta_l}(j,j_1)]\tilde{C}(j_1,j_2)[Y_{(k+1)\Delta_l}-Y_{k\Delta_l}](j_2)\Big),\\
T_3 & =  \sum_{j_1=1}^{d_x}[\alpha_{k\Delta_l}^N(j,j_1)-\alpha_{k\Delta_l}(j,j_1)]\omega_k^i(j_1),
\end{align*}
and we recall that $\tilde{C}=C^{\top} R_2^{-1}$. So as before it is enough to bound the $\mathbb{L}_q-$norms of $T_1,T_2, T_3$. 

The case of $T_1$ follows
by using 
\begin{align}
\nonumber
B_{k\Delta_l}^N(j,j_1)\xi_{k\Delta_l}^i(j_1) -B_{k\Delta_l}(j,j_1)\zeta_{k\Delta_l}^i(j_1)   &= 
\Big(\mathbb{I}_{\{j\}}(j_1) + A(j,j_1)\Delta_l - \sum_{j_2=1}^{d_x}P_{k\Delta_l}^N(j,j_2)S(j_2,j_1)\Delta_l\Big)\\ & \times [\xi_{k\Delta_l}^i(j_1)- \zeta_{k\Delta_l}^i(j_1)] \nonumber
\\ &+ \label{eq:rem_thm1}
 \sum_{j_2=1}^{d_x}[P_{k\Delta_l}^N(j,j_2)-P_{k\Delta_l}(j,j_2)]S(j_2,j_1)\Delta_l\zeta_{k\Delta_l}^i(j_1).
 \end{align}
Thus using Minkowski inequality several times
\begin{align*}
\mathbb{E}[|T_1|^q]^{1/q} &\leq \sum_{j_1=1}^{d_x}\Big(\mathbb{E}\Big[\Big|
\Big(\mathbb{I}_{\{j\}}(j_1) + A(j,j_1)\Delta_l - \sum_{j_2=1}^{d_x}P_{k\Delta_l}^N(j,j_2)S(j_2,j_1)\Delta_l\Big)[\xi_{k\Delta_l}^i(j_1)-\zeta_{k\Delta_l}^i(j_1)]
\Big|^q\Big]^{1/q} \\
&+ \mathbb{E}\Big[\Big|
\sum_{j_2=1}^{d_x}[P_{k\Delta_l}^N(j,j_2)-P_{k\Delta_l}(j,j_2)]S(j_2,j_1)\Delta_l\zeta_{k\Delta_l}^i(j_1)
\Big|^q\Big]^{1/q}\Big).
\end{align*}
{The first term} on the R.H.S.~can be dealt with using Cauchy-Schwarz and Lemma \ref{lem:xi_lq}.
For the second term on the R.H.S.~one can use Minkowski for the sum over $j_2$, then Cauchy-Schwarz, Lemma \ref{lem:p_lq} and properties of Gaussian random variables (recall $\zeta_{k\Delta_l}^i|\mathscr{F}_{k\Delta_l}\sim\mathcal{N}_{d_x}(m_{k\Delta_l},P_{k\Delta_l})$, so one can deduce that {$\mathbb{E}[|\zeta_{k\Delta_l}^i(j)|^q]^{1/q}\leq\mathsf{C}$, noting also the recursion \eqref{eq:iid_mean_rec}}) combined with \eqref{eq:bound_deter_cov}.
Thus, we have the upper-bound
\begin{equation}\label{eq:prop_1}
\mathbb{E}[|T_1|^q]^{1/q} \leq \mathsf{C}\Big(\frac{1}{\sqrt{N}} + \max_{j_1\in\{1,\dots,d_x\}}\mathbb{E}\Big[\Big|[\xi_{k\Delta_l}^i(j_1)-\zeta_{k\Delta_l}^i(j_1)]\Big|^{2q}\Big]^{1/(2q)}\Big).
\end{equation}

For $T_2$ by Minkowski
$$
\mathbb{E}[|T_2|^q]^{1/q} \leq \sum_{j_2=1}^{d_y}\sum_{j_1=1}^{d_x}\mathbb{E}[|[P_{k\Delta_l}^N(j,j_1)-P_{k\Delta_l}(j,j_1)]\tilde{C}(j_1,j_2)[Y_{(k+1)\Delta_l}-Y_{k\Delta_l}](j_2)|^{q}]^{1/q},
$$
using Cauchy-Schwarz and Lemmata \ref{lem:p_lq} and \ref{lem:y_inc}
\begin{equation}\label{eq:prop_2}
\mathbb{E}[|T_2|^q]^{1/q} \leq \frac{\mathsf{C}}{\sqrt{N}}.
\end{equation}

For $T_3$ on inspection of  \eqref{eq:enkf_ps} and \eqref{eq:enkf_is}, as well as letting $\bar{C}=\tilde{C}R_2^{1/2}$, we have
$$
T_3 = \Delta_l^{1/2}\sum_{j_2=1}^{d_x}\sum_{j_1=1}^{d_y}[P_{k\Delta_l}^N(j,j_1)-P_{k\Delta_l}(j,j_1)]\bar{C}(j_2,j_1)\tilde{\omega}_k^i(j_1),
$$
where independently of all other random variables $\tilde{\omega}_k^i\sim\mathcal{N}_{d_y}(0,I)$. Then by applying Minkowski 
$$
\mathbb{E}[|T_3|^q]^{1/q} \leq \sum_{j_2=1}^{d_x}\sum_{j_1=1}^{d_y}|\bar{C}(j_2,j_1)|\mathbb{E}[|P_{k\Delta_l}^N(j,j_1)-P_{k\Delta_l}(j,j_1)|^{q}]^{1/q}
\mathbb{E}[|\tilde{\omega}_k^i(j_1)|^q]^{1/q}.
$$
Then by Lemma \ref{lem:p_lq} and properties of Gaussian random variables we have
\begin{equation}\label{eq:prop_3}
\mathbb{E}[|T_3|^q]^{1/q} \leq \frac{\mathsf{C}}{\sqrt{N}}.
\end{equation}
We remark that indeed one can prove
\begin{equation}\label{eq:rem_thm}
\mathbb{E}[|T_3|^q]^{1/q} \leq \frac{\mathsf{C}\Delta_l^{1/2}}{\sqrt{N}}.
\end{equation}
which we shall use later on.
Combining \eqref{eq:prop_1}-\eqref{eq:prop_3} we have proved that
\begin{align*}
\max_{j\in\{1,\dots,d_x\}}\mathbb{E}\Big[\Big|\xi_{(k+1)\Delta_l}^i(j)-\zeta_{(k+1)\Delta_l}^i(j)\Big|^q\Big]^{1/q}  \leq \mathsf{C}\Big(\frac{1}{\sqrt{N}} + \max_{j_1\in\{1,\dots,d_x\}}\mathbb{E}\Big[\Big|[\xi_{k\Delta_l}^i(j_1)-\zeta_{k\Delta_l}^i(j_1)\Big|^{2q}\Big]^{1/(2q)}\Big).
\end{align*}
Applying the above recursion back to time zero one can conclude the proof.

\end{proof}

%
%

\section{Results for the Discretized Kalman-Bucy Diffusion}\label{app:en_euler}

We begin with some definitions.
For $(k,l)\in\mathbb{N}_0\times\mathbb{N}_0$
\begin{equation}\label{eq:deter_cov_evol}
P_{(k+1)\Delta_{l}}^l = P_{k\Delta_l}^l + \textrm{Ricc}(P_{k\Delta_l}^l)\Delta_l +  \textrm{SRicc}(P_{k\Delta_l}^l)\Delta_l^2,
\end{equation}
and $P_0^l=\mathcal{P}_0$. 
Note that one also has for $(k,l)\in\mathbb{N}_0\times\mathbb{N}_0$
\begin{equation}\label{eq:deter_mean_evol}
m_{(k+1)\Delta_l}^l =  m_{k\Delta_l}^l + Am_{k\Delta_l}\Delta_l + U_{k\Delta_l}^l\Big(
[Y_{(k+1)\Delta_l}-Y_{k\Delta_l}] -Cm_{k\Delta_l}^l\Delta_l\Big),
\end{equation}
and $m_0^l=\mathcal{M}_0$. These are the discretized Riccati and Kalman-Bucy equations that are satisfied by the Euler discretization of \eqref{eq:non-lin}, as defined
in \eqref{eq:iid_mean_rec}-\eqref{eq:iid_cov_rec}. The notation of the equations are simply modified for the multilevel context that will be considered later on.

{We first consider analyzing the the strong convergence, which is the strong error, associated with the Euler discretization} of \eqref{eq:non-lin}, for $(l,k)\in\mathbb{N}_0^2$:
\begin{align*}
\overline{X}_{(k+1)\Delta_l}^l & = \overline{X}_{k\Delta_l}^l + A\overline{X}_{k\Delta_l}^l\Delta_l + R_1^{1/2}[\overline{W}_{(k+1)\Delta_l}-\overline{W}_{k\Delta_l}] + U_{k\Delta_l}^l\Big([Y_{((k+1)\Delta_l)}-Y_{(k\Delta_l)}]  \\&- 
\Big[C\overline{X}_{k\Delta_l}^l\Delta_l+R_2^{1/2}[\overline{V}_{(k+1)\Delta_l}-\overline{V}_{k\Delta_l}]\Big]\Big),
\end{align*}
where for each $l\in\mathbb{N}_0$, $\overline{X}_{0}^l=\overline{X}_{0}\sim\mathcal{N}_{d_x}(\mathcal{M}_0,\mathcal{P}_0)$.
Note that the Markov chain for $i\in\{1,\dots,N\}$, $\zeta_{0}^i,\zeta_{\Delta_l}^i,\dots$ is simply an i.i.d.~realization of this Euler discretization.
We set $\tau_t^l={\lfloor\tfrac{t}{\Delta_l}\rfloor\Delta_l}$, $t\in\mathbb{R}^+$. Note that one can consider $P_{(k+1)\Delta_{l}}^l$ for each continuous time $t\geq 0$ by a simple integral representation for $(j_1,j_2)\in\{1,\dots,d_x\}^2$:
$$
P_{t}^l(j_1,j_2) = \mathcal{P}_0(j_1,j_2) + \int_{0}^t\Big(\textrm{Ricc}(P_{{\tau^l_s}}^l)(j_1,j_2) + \textrm{SRicc}(P_{{\tau^l_s}}^l)(j_1,j_2)\Delta_l\Big)ds.
$$
{For the rest of the appendix, we shall additionally assume that the matrices $(R_1^{1/2} AR_1^{1/2},\dots,A^{d_x-1}R_1^{1/2})$ and $(C, CA,\dots,CA^{d_x-1})^{\top}$ both have rank $d_x$. This ensures that there exists a $\mathsf{C}$ such that (Bucy's Theorem, \cite{bucy})
\begin{equation}\label{eq:p_upper_cont}
\max_{(j_1,j_2)\in\{1,\dots,d_x\}^2}\sup_{s\in[0,T]}|\mathcal{P}_{s}(j_1,j_2)|\leq\mathsf{C}.
\end{equation}
}

{This section consists of two results. The first is a proof concerning the convergence of $P_{\tau_t^l}^l$ to $\mathcal{P}_{t}$ as $l$ grows.
Given this result, we then prove the strong error associated to the Euler discretization of \eqref{eq:non-lin}.}

\begin{lem}\label{lem:p_disc}
For any $T\in\mathbb{N}$ fixed and $t\in[0,T]$ there exists a $\mathsf{C}<+\infty$ such that for any $(l,j_1,j_2)\in\mathbb{N}_0\times\{1,\dots,d_x\}^2$:
$$
\Big|\mathcal{P}_{t}(j_1,j_2)-P_{\tau_t^l}^l(j_1,j_2)\Big| \leq \mathsf{C}\Delta_l.
$$
\end{lem}

\begin{proof}
We have
\begin{equation}\label{eq:p_disc_5}
\mathcal{P}_{t}(j_1,j_2)-P_{\tau_t^l}^l(j_1,j_2) = \mathcal{P}_{t}(j_1,j_2)-P_{t}^l(j_1,j_2) + P_{t}^l(j_1,j_2)-P_{\tau_t^l}^l(j_1,j_2).
\end{equation}
By \eqref{eq:bound_deter_cov} one can deduce that 
\begin{equation}\label{eq:p_disc_4}
\max_{(j_1,j_2)\in\{1,\dots,d_x\}^2}|P_{t}^l(j_1,j_2)-P_{\tau_t^l}^l(j_1,j_2)| \leq \mathsf{C}\Delta_l.
\end{equation}
Hence we only consider
\begin{equation}\label{eq:p_disc_1}
\mathcal{P}_{t}(j_1,j_2)-P_{t}^l(j_1,j_2) = \int_{0}^t\Big(\textrm{Ricc}(\mathcal{P}_{s})(j_1,j_2)-\textrm{Ricc}(P_{\tau_s^l}^l)(j_1,j_2)\Big)ds - 
\int_{0}^t\textrm{SRicc}(P_{\tau_s^l}^l)(j_1,j_2)ds\Delta_l.
\end{equation}
Using {standard} calculations along with \eqref{eq:bound_deter_cov} one can deduce that 
\begin{equation}\label{eq:p_disc_2}
\max_{(j_1,j_2)\in\{1,\dots,d_x\}^2}\Big|\int_{0}^t\textrm{SRicc}(P_{\tau_s^l}^l)(j_1,j_2)ds\Delta_l\Big| \leq \mathsf{C}\Delta_l,
\end{equation}
so we need only consider the first term on the R.H.S.~of \eqref{eq:p_disc_1}.

The following decomposition holds
$$
\textrm{Ricc}(\mathcal{P}_{s})(j_1,j_2)-\textrm{Ricc}(P_{\tau_s^l}^l)(j_1,j_2) = \sum_{j=1}^3 T_j,
$$
where
\begin{eqnarray*}
T_1 & = & \sum_{j_3=1}^{d_x} A(j_1,j_3)[\mathcal{P}_{s}(j_3,j_2)-P_{\tau_s^l}^l(j_3,j_2)],\\
T_2 & = & \sum_{j_3=1}^{d_x} A(j_3,j_2)[\mathcal{P}_{s}(j_1,j_3)-P_{\tau_s^l}^l(j_1,j_3)],\\
T_3 & = & \sum_{j_3=1}^{d_x} \sum_{j_4=1}^{d_x} S(j_3,j_4)[\mathcal{P}_{s}(j_1,j_3)\mathcal{P}_{s}(j_4,j_2)-P_{\tau_s^l}^l(j_1,j_3)P_{\tau_s^l}^l(j_4,j_2)].
\end{eqnarray*}
So it suffices to control each term in the sum.

For both $T_1$ and $T_2$ it is straightforward to deduce that for $j\in\{1,2\}$
$$
T_j \leq \mathsf{C}\max_{(j_1,j_2)\in\{1,\dots,d_x\}^2}|\mathcal{P}_{s}(j_1,j_2)-P_{\tau_s^l}^l(j_1,j_2)|.
$$
For $T_3$ we have that
$$
T_3 = \sum_{j_3=1}^{d_x} \sum_{j_4=1}^{d_x} S(j_3,j_4)[(\{\mathcal{P}_{s}(j_1,j_3)-P_{\tau_s^l}^l(j_1,j_3)\}\mathcal{P}_{s}(j_4,j_2)+
P_{\tau_s^l}^l(j_1,j_3)\{\mathcal{P}_{s}(j_4,j_2)-P_{\tau_s^l}^l(j_4,j_2)\}].
$$
Then by using \eqref{eq:bound_deter_cov}, coupled {with \eqref{eq:p_upper_cont}} we have 
$$
T_3 \leq \mathsf{C}\max_{(j_1,j_2)\in\{1,\dots,d_x\}^2}|\mathcal{P}_{s}(j_1,j_2)-P_{\tau_l^s}^l(j_1,j_2)|.
$$
Thus we have proved that
\begin{align}\label{eq:p_disc_3}
\max_{(j_1,j_2)\in\{1,\dots,d_x\}^2}|\textrm{Ricc}(\mathcal{P}_{s})(j_1,j_2)-\textrm{Ricc}(P_{\tau_s^l}^l)(j_1,j_2)|  \leq  \mathsf{C}\max_{(j_1,j_2)\in\{1,\dots,d_x\}^2}|\mathcal{P}_{s}(j_1,j_2)-P_{\tau_s^l}^l(j_1,j_2)|. 
\end{align}

Combining \eqref{eq:p_disc_5}-\eqref{eq:p_disc_3}, we have proved that
\begin{align*}
\max_{(j_1,j_2)\in\{1,\dots,d_x\}^2}\Big|\mathcal{P}_{t}(j_1,j_2)-P_{\tau_t^l}^l(j_1,j_2)\Big|  \leq \mathsf{C}\Big(
\int_{0}^t \max_{(j_1,j_2)\in\{1,\dots,d_x\}^2}|\mathcal{P}_{s}(j_1,j_2)-P_{\tau_s^l}^l(j_1,j_2)| ds+\Delta_l\Big).
\end{align*}
On applying Gronwall's lemma, the result is concluded.

\end{proof}

Below we will use the $C_p-$inequality several times and this is as follows. For two real-valued random variables
$X$ and $Y$ defined on the same probability space, with expectation operator $\mathbb{E}$, suppose that for some fixed $p\in(0,\infty)$, $\mathbb{E}[|X|^p]$
and $\mathbb{E}[|Y|^p]$ are finite, then the $C_p-$inequality is
$$
\mathbb{E}[|X+Y|^p] \leq C_p\Big(\mathbb{E}[|X|^p]+ \mathbb{E}[|Y|^p]\Big),
$$
where $C_p=1$, if $p\in(0,1)$ and $C_p=2^{p-1}$ for $p\in[1,\infty)$.

\begin{lem}\label{lem:strong_error}
For any $T\in\mathbb{N}$ fixed and $t\in[0,T-1]$ there exists a $\mathsf{C}<+\infty$ such that for any $(l,j,k_1)\in\mathbb{N}_0\times\{1,\dots,d_x\}\times\{0,1,\dots,\Delta_{l}^{-1}\}$:
$$
\mathbb{E}\Big[\Big(\overline{X}_{t+k_1\Delta_{l}}(j)-\overline{X}_{t+k_1\Delta_{l}}^l(j)\Big)^2\Big] \leq \mathsf{C}\Delta_l^2.
$$
\end{lem}

\begin{proof}
By using the $C_2-$inequality four times, one has the upper-bound
$$
\mathbb{E}\Big[\Big(\overline{X}_{t+k_1\Delta_{l}}(j)-\overline{X}_{t+k_1\Delta_{l}}^l(j)\Big)^2\Big] \leq \mathsf{C}\sum_{j=1}^4 T_j,
$$
where
\begin{align*}
T_1 & = \mathbb{E}\Big[\Big(\int_0^{t+k_1\Delta_{l}}\sum_{j_1=1}^{d_x} A(j,j_1)[\overline{X}_{s}(j)-\overline{X}_{\tau_s^l}^l(j)]ds\Big)^2\Big],\\
T_2 & = \mathbb{E}\Big[\Big(\sum_{j_1=1}^{d_x}\sum_{j_2=1}^{d_y}\int_{0}^{t+k_1\Delta_{l}}[\mathcal{P}_{s}(j,j_1)-P_{\tau_s^l}^l(j,j_1)]\tilde{C}(j_1,j_2)dY_s(j_2)
\Big)^2\Big],\\
T_3 &= \mathbb{E}\Big[\Big(\sum_{j_1=1}^{d_x}\sum_{j_2=1}^{d_x}\int_{0}^{t+k_1\Delta_{l}}[\mathcal{P}_{s}(j,j_1)\hat{C}(j_1,j_2)\overline{X}_s(j_2)
-P^l_{\tau_s^l}(j,j_1)\hat{C}(j_1,j_2)\overline{X}_{\tau_s^l}^l(j_2)]ds
\Big)^2\Big],\\
T_4 & = \mathbb{E}\Big[\Big(\sum_{j_1=1}^{d_x}\sum_{j_2=1}^{d_y}\int_{0}^{t+k_1\Delta_{l}}[\mathcal{P}_{s}(j,j_1)-P_{\tau_s^l}^l(j,j_1)]\tilde{C}(j_1,j_2)d\overline{V}_s(j_2)
\Big)^2\Big],
\end{align*}
with $\tilde{C}=C^{\top}R_2^{-1/2}$.
We shall focus on upper-bounding each of the terms $T_1,\dots,T_4$ and summing the bounds.

For $T_1$, $d_x-$applications of the $C_2-$inequality along with Jensen's inequality yields
\begin{equation}\label{eq:strong_error1}
T_1 \leq \mathsf{C}\int_0^{t+k_1\Delta_{l}}\max_{j\in\{1,\dots,d_x\}}\mathbb{E}\Big[\Big(\overline{X}_{s}(j)-\overline{X}_{\tau_s^l}^l(j)\Big)^2\Big]ds.
\end{equation}

For $T_2$, using \eqref{eq:data} and the $C_2-$inequality we have
$$
T_2 \leq T_5 + T_6,
$$
where
\begin{align*}
T_5 &= 
 2\mathbb{E}\Big[\Big(\sum_{j_1=1}^{d_x}\sum_{j_2=1}^{d_y}\sum_{j_3=1}^{d_x}\int_{0}^{t+k_1\Delta_{l}}[\mathcal{P}_{s}(j,j_1)-P_{\tau_s^l}^l(j,j_1)]\tilde{C}(j_1,j_2)
C(j_2,j_3)X_s(j_3)ds
\Big)^2\Big], \\
T_6 & =  2
\mathbb{E}\Big[\Big(\sum_{j_1=1}^{d_x}\sum_{j_2=1}^{d_y}\sum_{j_3=1}^{d_y}\int_{0}^{t+k_1\Delta_{l}}[\mathcal{P}_{s}(j,j_1)-P_{\tau_s^l}^l(j,j_1)]\tilde{C}(j_1,j_2)
R_2^{1/2}(j_2,j_3)dV_s(j_3)
\Big)^2\Big].
\end{align*}
For $T_5$ using $d_x^2d_y-$applications of the $C_2-$inequality along with Jensen's inequality yields
$$
T_5 \leq \mathsf{C}
\int_{0}^{t+k_1\Delta_{l}}
\max_{(j_1,j_2)\in\{1,\dots,d_x\}^2}\Big|\mathcal{P}_{s}(j,j_1)-P_{\tau_s^l}^l(j,j_1)\Big|^2\max_{j\in\{1,\dots,d_x\}}\mathbb{E}[X_s(j)^2]ds.
$$
Then applying Lemma \ref{lem:p_disc}, using the fact that $\max_{j\in\{1,\dots,d_x\}}\mathbb{E}[X_{s}(j)^2]\leq \mathsf{C}$
and that one is integrating over a finite domain (so the constants that depend on $t$ are upper-bounded), we have
$$
T_5 \leq \mathsf{C}\Delta_l^2.
$$
For $T_6$ using $d_xd_y^2-$applications of the $C_2-$inequality along with the Ito Isometry formula yields
$$
T_6 \leq \mathsf{C}
\int_{0}^{t+k_1\Delta_{l}}
\max_{(j_1,j_2)\in\{1,\dots,d_x\}^2}\Big|\mathcal{P}_{s}(j,j_1)-P_{\tau_s^l}^l(j,j_1)\Big|^2ds.
$$
Applying  Lemma \ref{lem:p_disc} and the above argument for $T_5$ gives
$$
T_6 \leq \mathsf{C}\Delta_l^2.
$$
Therefore we have
\begin{equation}\label{eq:strong_error2}
T_2 \leq \mathsf{C}\Delta_l^2.
\end{equation}

For $T_3$ $d_x^2-$applications of the $C_2-$inequality along with Jensen's inequality gives the upper-bound
$$
T_3 \leq \mathsf{C}(T_7 + T_8),
$$
where
\begin{align*}
T_7 &=  \int_{0}^{t+k_1\Delta_{l}}
\max_{(j_1,j_2)\in\{1,\dots,d_x\}^2}\Big|\mathcal{P}_{s}(j,j_1)-P_{\tau_s^l}^l(j,j_1)\Big|^2\max_{j\in\{1,\dots,d_x\}}\mathbb{E}[X_s(j)^2]ds, \\
T_8 &= \int_0^{t+k_1\Delta_{l}}
\max_{(j_1,j_2)\in\{1,\dots,d_x\}^2}|P_{\tau_s^l}^l(j,j_1)|^2
\max_{j\in\{1,\dots,d_x\}}\mathbb{E}\Big[\Big(\overline{X}_{s}(j)-\overline{X}_{\tau_s^l}^l(j)\Big)^2\Big]ds.
\end{align*}
For $T_7$, one can simply use the same argument as for $T_5$ and for $T_8$, by \eqref{eq:bound_deter_cov} one has the upper-bound
$$
T_8 \leq \mathsf{C}\int_0^{t+k_1\Delta_{l}}\max_{j\in\{1,\dots,d_x\}}\mathbb{E}\Big[\Big(\overline{X}_{s}(j)-\overline{X}_{\tau_s^l}^l(j)\Big)^2\Big]ds,
$$
thus we have
\begin{equation}\label{eq:strong_error3}
T_3 \leq \mathsf{C}\Big(\int_0^{t+k_1\Delta_{l}}\max_{j\in\{1,\dots,d_x\}}\mathbb{E}\Big[\Big(\overline{X}_{s}(j)-\overline{X}_{\tau_s^l}^l(j)\Big)^2\Big]ds+\Delta_l^2\Big).
\end{equation}

For $T_4$, one can use almost the same argument as for $T_6$ to give
\begin{equation}\label{eq:strong_error4}
T_4 \leq \mathsf{C}\Delta_l^2.
\end{equation}
Combining \eqref{eq:strong_error1}-\eqref{eq:strong_error4} gives the upper-bound
\begin{align*}
\max_{j\in\{1,\dots,d_x\}}\mathbb{E}\Big[\Big(\overline{X}_{t+k_1\Delta_{l}}(j)-\overline{X}_{t+k_1\Delta_{l}}^l(j)\Big)^2\Big] \leq 
\mathsf{C}\Big(\int_0^{t+k_1\Delta_{l}}\max_{j\in\{1,\dots,d_x\}}\mathbb{E}\Big[\Big(\overline{X}_{s}(j)-\overline{X}_{\tau_s^l}^l(j)\Big)^2\Big]ds+\Delta_l^2\Big).
\end{align*}
On applying Gronwall's lemma, the result is concluded.

\end{proof}

\section{Results for the i.i.d.~Multilevel d-EnBKF}\label{app:ml_res}

The analysis of this section concerns the system decsribed in \eqref{eq:iid1}-\eqref{eq:iid2} and the resulting multilevel estimator.

\subsection{Variance}

We recall some conventions. The $d_x-$dimensional Gaussian measure with mean $m_{k\Delta_l}^l$ and covariance $P_{k\Delta_l}^l$ (see \eqref{eq:deter_cov_evol}-\eqref{eq:deter_mean_evol}.) is written as $\eta_{k\Delta_l}^l$. Also recall $e(x)=x$. We write the empirical measure of the particle system \eqref{eq:enkf_is} at a time $t$ as
$\hat{\eta}_{t+k_1\Delta_l}^{N,l}$. Finally, we used the notation:  for a $d_x-$dimensional vector $x$ denote $\|x\|_2=(\sum_{j=1}^{d_x}x(j)^2)^{1/2}$.

\begin{proposition}\label{prop:var_term1}
{For any $T\in\mathbb{N}$ fixed and $t\in[0,T-1]$ there exists a $\mathsf{C}<+\infty$ such that for any $(l,N,k_1)\in\mathbb{N}_0\times\{2,3,\dots\}\times\{0,1,\dots,\Delta_l^{-1}\}$:
$$
\mathbb{E}\Big[\Big\|[\hat{\eta}_{t+k_1\Delta_l}^{N,l}-\eta_{t+k_1\Delta_l}](e)\Big\|_2^2\Big] \leq \mathsf{C}\Big(\frac{1}{N}+\Delta_l^2\Big).
$$}
\end{proposition}

\begin{proof}
Using the $C_2-$inequality one can has
\begin{align}
\mathbb{E}\Big[\Big\|[\hat{\eta}_{t+k_1\Delta_l}^{N,l}-\eta_{t+k_1\Delta_l}](e)\Big\|_2^2\Big] &\leq 
\label{eq:first_prop_1}
\mathsf{C}\Big(
\mathbb{E}\Big[\Big\|[\hat{\eta}_{t+k_1\Delta_l}^{N,l}-\eta_{t+k_1\Delta_l}^l](e)\Big\|_2^2\Big] + 
\mathbb{E}\Big[\Big\|[\eta_{t+k_1\Delta_l}^{l}-\eta_{t+k_1\Delta_l}](e)\Big\|_2^2\Big] 
\Big).
\end{align}
{The first term on the R.H.S.~can be controlled by standard results for i.i.d.~sampling (recall that $\zeta_{t+k_1\Delta_l}^i|\mathscr{F}_{t+k_1\Delta_l}$ are i.i.d.~Gaussian with mean $m_{t+k_1\Delta_l}$ and covariance $P_{t+k_1\Delta_l}$), that is
\begin{equation}\label{eq:first_prop_2}
\mathbb{E}\Big[\Big\|[\eta_{t+k_1\Delta_l}^{N,l}-\eta_{t+k_1\Delta_l}^l](e)\Big\|_2^2\Big] \leq \frac{\mathsf{C}}{N}.
\end{equation}
Note that it is crucial that \eqref{eq:bound_deter_cov} holds, otherwise the upper-bound can explode as a function of $l$.}
For the right-most term on the R.H.S.~of \eqref{eq:first_prop_1} by Jensen's inequality and Lemma \ref{lem:strong_error}:
\begin{equation}\label{eq:main_bias}
\mathbb{E}\Big[\Big\|[\eta_{t+k_1\Delta_l}^{l}-\eta_{t+k_1\Delta_l}](e)\Big\|_2^2\Big]  \leq \mathsf{C}\Delta_l^2.
\end{equation}
So the proof can be concluded by combining \eqref{eq:first_prop_1} with \eqref{eq:first_prop_2} and \eqref{eq:main_bias}.

\end{proof}

\begin{proposition}\label{prop:var_term2}{
For any $T\in\mathbb{N}$ fixed and $t\in[0,T-1]$ there exists a $\mathsf{C}<+\infty$ such that for any $(l,N,k_1)\in\mathbb{N}\times\{2,3,\dots\}\times\{0,1,\dots,\Delta_{l-1}^{-1}\}$:
$$
\mathbb{E}\Big[\Big\|[\hat{\eta}_{t+k_1\Delta_{l-1}}^{N,l}-\hat{\eta}_{t+k_1\Delta_{l-1}}^{N,l-1}](e) - [\eta_{t+k_1\Delta_{l-1}}^{l}-\eta_{t+k_1\Delta_{l-1}}^{l-1}](e)\Big\|_2^2\Big] \leq \frac{\mathsf{C}\Delta_l}{N}.
$$}
\end{proposition}

\begin{proof}
One can use standard properties of i.i.d.~empirical averages along with Lemma \ref{lem:strong_error} to give
$$
\mathbb{E}\Big[\Big\|[\hat{\eta}_{t+k_1\Delta_{l-1}}^{N,l}-\hat{\eta}_{t+k_1\Delta_{l-1}}^{N,l-1}](e) - [\eta_{t+k_1\Delta_{l-1}}^{l}-\eta_{t+k_1\Delta_{l-1}}^{l-1}](e)\Big\|_2^2\Big] \leq \frac{\mathsf{C}\Delta_l}{N}.
$$

\end{proof}

%
\subsection{Proof of Theorem \ref{theo:main_theo}}
\begin{proof}
Noting \eqref{eq:main_est} one has
$$
[\hat{\eta}_t^{ML}-\eta_t](e) = [\hat{\eta}_t^{N_0,0}-\eta_t^0](e) + \sum_{l=1}^L [\hat{\eta}_t^{N_l,l} -\hat{\eta}_t^{N_l,l-1}-\eta_t^{l} -\eta_t^{l-1}](e) + [\eta_t^L-\eta_t](e).
$$
Thus, by using three applications of the $C_2-$inequality we have
\begin{align*}
\mathbb{E}\Big[\Big\|[\eta_t^{ML}-\eta_t](e)\Big\|_2^2\Big] &\leq \mathsf{C}\Big(
\mathbb{E}\Big[\Big\|[\hat{\eta}_t^{N_0,0}-\eta_t^0](e)\Big\|_2^2\Big] \\ &+ \mathbb{E}\Big[\Big\|\sum_{l=1}^L [\hat{\eta}_t^{N_l,l} -\hat{\eta}_t^{N_l,l-1}-\eta_t^{l} -\eta_t^{l-1}](e)\Big\|_2^2\Big]
+\mathbb{E}\Big[\Big\|[\eta_t^L-\eta_t](e)\Big\|_2^2\Big]
\Big).
\end{align*}
For the first term on the R.H.S.~one can use \eqref{eq:first_prop_2} and for the last term on the R.H.S.~\eqref{eq:main_bias}. For the middle term, one has
\begin{align*}
\Big\|\sum_{l=1}^L [\hat{\eta}_t^{N_l,l} -\hat{\eta}_t^{N_l,l-1}-\eta_t^{l} -\eta_t^{l-1}](e)\Big\|_2^2  &= 
\sum_{l=1}^L\sum_{j=1}^{d_x}\Big([\hat{\eta}_t^{N_l,l} -\hat{\eta}_t^{N_l,l-1}-\eta_t^{l} -\eta_t^{l-1}](e)\Big)(j)^2
 \\ &+  \sum_{l=1}^L\sum_{q=1}^L\mathbb{I}_{D^c}(l,q)\sum_{j=1}^{d_x}\Big([\hat{\eta}_t^{N_l,l} -\hat{\eta}_t^{N_l,l-1}-\eta_t^{l} -\eta_t^{l-1}](e)\Big)(j)\\ &\times  \Big([\hat{\eta}_t^{N_q,q} -\hat{\eta}_t^{N_q,q-1}-\eta_t^{q} -\eta_t^{q-1}](e)\Big)(j).
\end{align*}
Then using a combination of the independence of the coupled particle systems along with Proposition \ref{prop:var_term2} the proof can be concluded.

\end{proof}

\end{document}